\documentclass[10pt]{amsart}

\usepackage[all]{xy}
\usepackage{appendix}

\usepackage{mathpazo}

\usepackage{amsmath, amssymb, amsthm, latexsym, amscd}

\numberwithin{equation}{section}

\newtheorem{theorem}[subsection]{Theorem}
\newtheorem{prop}[subsection]{Proposition}
\newtheorem{lemma}[subsection]{Lemma}
\newtheorem{corollary}[subsection]{Corollary}

\theoremstyle{definition}
\newtheorem{define}[subsection]{Definition}

\newtheorem{examples}[subsection]{Examples}
\newtheorem{remark}[subsection]{Remark}

\title{Vertex operator algebras associated to modified regular representations of the Virasoro algebra}

\author{Igor Frenkel, Minxian Zhu}
\thanks{${}^*$The first author was partially supported by the NSF grant DMS-04574444.}

\address{Department of Mathematics, 10 Hillhouse Ave, Yale University, New Haven, CT 06520}
\email{frenkel@math.yale.edu}

\address{Department of Mathematics, 110 Frelinghuysen Rd, Rutgers University, Piscataway, NJ 08854}
\email{minxian@math.rutgers.edu}

\begin{document}
\maketitle

\begin{abstract}
We give an abstract construction, 
based on the Belavin-Polyakov-Zamolodchikov equations, 
of a family of vertex operator algebras 
of rank $26$
associated to the modified regular representations of the Virasoro algebra. 
The vertex operators are obtained from the 
products of intertwining operators for a pair of 
Virasoro algebras.
We explicitly determine the structure coefficients
that yield the axioms of VOAs.
In the process of our construction, 
we obtain new hypergeometric identities. 
\end{abstract}

\tableofcontents

\section{Introduction}

Two classes of vertex operator algebras associated to vacuum representations of affine and Virasoro algebras were constructed and studied in [FZ]. 
The regular representations of these vertex operator algebras yield well-known examples of two-dimensional conformal field theories
(see e.g. [DF]), 
but do not possess the structure of vertex operator algebras. 
It was shown in [FS] that one can modify this regular representation in a way that produces a VOA structure.
The construction in [FS] was based on the free bosonic realization and was explicitly verified only for the affine Lie algebra $\widehat {sl_2}$ and the Virasoro algebra. 
The VOA structure for a general affine Lie algebra was obtained in [Z] using the abstract theory of intertwining operators and the Knizhnik-Zamolodchikov equations for the correlation functions. 

In the present paper, 
we use a similar approach, 
now based on the Belavin-Polyakov-Zamolodchikov equations, 
to give another construction of the vertex operator algebras associated to the modified regular representations of the Virasoro algebra. 
We also explicitly compute the structure coefficients of this class of vertex operator algebras. 
The construction uses a pair of the Virasoro algebras $\text{Vir}_c, \text{Vir}_{\bar c}$ with central charges 
\begin{eqnarray}
c = 13 - 6 \varkappa - 6 \varkappa^{-1}, & &
\bar c = 13 + 6 \varkappa + 6 \varkappa^{-1}
\end{eqnarray}
for $\varkappa \in \mathbb C \backslash \mathbb Q$, 
so that the rank of our vertex operator algebras is always equal to 
$c + \bar c = 26$. 
The modified regular representation $W$ 
is comprised of 
the irreducible highest weight representations 
$L(\triangle(\lambda), c)$, 
$L(\bar \triangle(\lambda), \bar c)$, $\lambda\in \mathbb N$, 
of the two Virasoro algebras
$\text{Vir}_c, \text{Vir}_{\bar c}$
with the highest weights given by (1.1) and (1.2)
\begin{eqnarray}
\triangle(\lambda) = \frac{\lambda (\lambda+2)}{4 \varkappa} - \frac{\lambda}{2}, & &
\bar \triangle(\lambda) = - \frac{\lambda (\lambda+2)}{4 \varkappa} - \frac{\lambda}{2}
\end{eqnarray}
for $\lambda \in \mathbb N$. 
As a $\text{Vir}_c \oplus \text{Vir}_{\bar c}$-module, 
it has the form 
\begin{eqnarray}
W & = & \bigoplus_{\lambda \in \mathbb N} W(\lambda) 
= \bigoplus_{\lambda \in \mathbb N} 
L(\triangle(\lambda), c) \otimes 
L(\bar \triangle(\lambda), \bar c). 
\end{eqnarray}
Furthermore, 
$W$ is equipped with a VOA structure as follows: 
for any $ v \otimes v' \in W(\lambda)$, $u\otimes u' \in W(\mu)$, 
we have 
\begin{eqnarray}
Y( v \otimes v', z) u\otimes u' & = & 
\sum_{\nu \in \mathbb N} X_{\lambda \mu}^\nu
\Phi_{\lambda, \mu}^\nu (v, z) u \otimes \Psi_{\lambda, \mu}^\nu (v', z) u',
\end{eqnarray}
where 
$\Phi_{\lambda, \mu}^\nu (\cdot, z)$, $\Psi_{\lambda, \mu}^\nu (\cdot, z) $
are normalized intertwining operators of type 
$\left( \begin{array}{c}
L(\triangle(\nu), c) \\
L(\triangle(\lambda), c) \,\, L(\triangle(\mu), c)
\end{array}
\right)$
and 
$\left( \begin{array}{c}
L(\bar \triangle(\nu), \bar c) \\
L(\bar \triangle(\lambda), \bar c) \,\, L(\bar \triangle(\mu),\bar c)
\end{array}
\right)$
respectively. 
The structure constants $X_{\lambda \mu}^\nu$ vanish
unless $|\lambda - \mu| \leq \nu \leq \lambda + \mu$ 
and $\nu \equiv \lambda + \mu \, (\text{ mod } 2)$, 
in which case they are given by the following formula
\begin{eqnarray}
X_{\lambda \mu}^\nu & = &
\frac{ {\lambda \choose \ell} { \mu \choose \ell} }{ { \lambda + \mu - \ell +1 \choose \ell} }
\frac{1}{[ (\nu+2)^2 - \varkappa^2 ]  [ (\nu+3)^2 - \varkappa^2 ] \ldots 
 [ ( \nu + \ell +1 )^2 - \varkappa^2 ] }
\end{eqnarray}
where $\ell = \frac{\lambda + \mu -\nu}{2}$. 


The formula (1.5) clearly indicates the special role of integral values of $\varkappa$. 
It is an interesting problem to construct certain restricted versions of our vertex operator algebras in these cases. 
It is also an interesting question to rederive our formula (1.5) for the structure coefficients using the original construction of the vertex operator algebras in [FS] and the free bosonic realization of the intertwining operators. 

This article is organized as follows. 
In Section 2, 
we compute explicitly the bimodule associated to $L(\triangle(\lambda), c)$, $\lambda \in \mathbb N$, 
and then prove the fusion rules for this class of Virasoro representations. 
In Section 3, 
we recall the derivation of BPZ equations for the correlation functions of intertwining operators associated to 
$L(\triangle(1), c)$. 
We define the vertex operators $Y(v \otimes v', z)$ for $v \otimes v' \in W(1)$, 
and prove the locality using properties of the 
Gauss hypergeometric function. 
In Section 4, 
we use the reconstruction theorem 
for vertex operator algebras 
and induction on $\lambda$ 
to extend
formula (1.4) to a general $\lambda$. 
In the process of our analysis of the locality properties of vertex operators, 
we deduce (1.5) for the structure constants of our construction.  
Finally, in Section 5, 
we collect the hypergeometric identities used at the various steps of our construction. 

\section{Fusion rules}

The Virasoro algebra is the Lie algebra 
$$
\text{Vir} = \mathbb C L_n \oplus \mathbb C \underline c
$$
with commutation relations
$$
[L_m, L_n] = (m-n) L_{m+n} + \delta_{m+n, 0} \frac{m^3-m}{12} \underline c
$$
$$
[L_m, \underline c] = 0. 
$$
Set 
$$
L_- = \bigoplus_{n=1}^\infty \mathbb C L_{-n}, \qquad 
L_+ = \bigoplus_{n=1}^\infty \mathbb C L_{n}. 
$$

Given complex numbers $h$ and $c$, 
the Verma module $M(h, c)$ is a free $U(L_-)$-module 
generated by a highest weight vector $v= v_{h, c}$ 
satisfying 
$$
L_0 \cdot v = h v, \quad \underline c \cdot v = c v,  \quad \text{and} \quad L_n \cdot v = 0 \hspace{.05in}\text{ for } n >0. 
$$ 
A vector $u \in M(h, c)$ is called a singular (or null) vector 
if $L_+ \cdot u =0$ and $u$ is an eigenvector of $L_0$. 
Obviously, 
the highest weight vector $v_{h, c}$ is singular.  
Feigin and Fuchs in [FF] found all the singular vectors of the Verma modules and 
gave complete descriptions of their structure. 
When $h=0$, 
the vector $L_{-1} v_{0, c}$ is also singular. 
We denote by $M_c$ the quotient module of $M(0, c)$ 
divided by the submodule generated by $L_{-1} v_{0, c}$. 
It is shown in [FZ] that $M_c$ is a vertex operator algebra of rank $c$ 
with Virasoro element  $\omega = L_{-2} \mathbf 1$. 
The Verma module $M(h, c)$ admits a unique maximal proper submodule $J(h, c)$
with irreducible quotient $$L(h, c) = M(h, c) /J(h, c). $$
The modules $M(h, c), L(h, c)$ are all representations of $M_c$.

We fix $\varkappa \notin \mathbb Q$ and set 
$$
c = 13 - 6\varkappa -6\varkappa^{-1}, \qquad 
\triangle(\lambda) = \frac{\lambda(\lambda+2)}{4\varkappa} - \frac{\lambda}{2}, 
\qquad \text{for  } \lambda \in \mathbb Z. 
$$

\begin{lemma}\label{sing}
For $\lambda \in \mathbb N$, 
the maximal proper submodule $J(\triangle(\lambda), c)$ 
of $M(\triangle(\lambda), c)$ 
is generated by a single singular vector 
$S_{1, n+1} v_{\triangle(n), c}$
where 
$$
S_{1, n+1} = \sum_{n+1 = p_1 + \cdots + p_r, p_i \geq 1} 
\frac{n!^2}{\prod_{i=1}^{r-1} (p_1 + \cdots p_i)( p_{i+1} + \cdots p_r)} (-\varkappa^{-1})^{n + 1 - r}
L_{-p_1} \cdots L_{-p_r} . 
$$
For $\lambda <0, \lambda \in \mathbb Z$, 
the Verma module $M( \triangle(\lambda), c)$ is irreducible. 
Moreover, 
$J(\triangle(\lambda), c) \cong M(\triangle(-\lambda-2), c)$ for $\lambda \geq 0$. 
\end{lemma}

\begin{proof}
See [FF], [BS]. 
\end{proof}


\begin{examples} \label{null}
For $\lambda=0$, 
$J(0, c)$ is generated by $L_{-1} v_{0, c}$, 
so $L(0, c) = M_c$. 
For $\lambda=1$, 
$J(\frac{3}{4\varkappa} - \frac{1}{2}, c)$ is generated by 
$$
(L_{-1}^2 - \varkappa^{-1} L_{-2}) v_{\frac{3}{4\varkappa} - \frac{1}{2}, c}.
$$
For $\lambda=2$, 
$J(\frac{2}{\varkappa}-1, c)$ is generated by 
$$
(L_{-1}^3 - 2 \varkappa^{-1} L_{-1}L_{-2} - 2 \varkappa^{-1} L_{-2}L_{-1} + 4 \varkappa^{-2} L_{-3})
v_{\frac{2}{\varkappa}-1, c}. 
$$

\end{examples}


For every vertex operator algebra $V$, 
Y. Zhu defined in [Zh] an associative algebra $A(V)$
and established a one-to-one correspondence between irreducible representations of $V$ and irreducible representations of $A(V)$. 
In general, 
given a $V$-module $M$,
he defined a space $A(M)$ which becomes an $A(V)$-bimodule. 
Frenkel and Y. Zhu used the $A(V)$-theory 
to prove the rationality of vertex operator algebras 
associated to the irreducible vacuum representations of affine Lie algebras 
in positive integral levels 
and computed the fusion rules between their irreducible modules. 
We recall the definitions of $A(V)$ and $A(M)$ from [Zh] 
and use them to determine fusion rules 
among certain irreducible representations of the Virasoro algebra.

\begin{define}
Let $V$ be a vertex operator algebra. 
We define a multiplication in $V$ as follows: 
for any homogeneous $a\in V$, let 
$$
a \ast b = \text{Res}_z \left( Y(a, z) \frac{(z+1)^{\text{deg} a}}{z} b \right). 
$$
Moreover, 
let $O(V)$ be the linear span of elements 
$$
\text{Res}_z \left(Y(a, z) \frac{(z+1)^{\text{deg} a}}{z^2} b\right).
$$
As a vector space, 
$A(V) = V/O(V)$ and the operation $\ast$ induces a multiplication on $A(V)$ making it an associative algebra. 
We denote by $[a]$ the image of $a$ in $A(V)$. 
Then $[\mathbf 1]$ is the identity element and 
$[\omega]$ is central in $A(V)$. 
\end{define}







\begin{lemma} \label{a}
There is an isomorphism of associative algebras: 
$A(M_c) \cong \mathbb C[x]$ 
given by 
$[\omega]^n \mapsto x^n$.  
\end{lemma}

\begin{proof}
See [DMZ, Proposition 3.1],  also [W, Lemma 4.1], [L, Proposition 2.15]. 
\end{proof}

Let $M =\bigoplus_{n=0}^\infty M_n$ be a $V$-module. 
For any homogeneous $a \in V$, 
the vertex operator $a(\text{deg} a-1)$ preserves the grading on $M$. 

\begin{lemma}
The top level $M_0$ becomes an $A(V)$-module where $[a]$ acts by $a(\text{deg} a-1)$. 
\end{lemma}

\begin{lemma} \label{aa}
For $V = M_c$ and $M =  M(h, c)$, 
we have $M_0 \cong \mathbb C_h$, 
where $\mathbb C_h$ is the $ \mathbb C[x]$-module such that $x$ acts as multiplication by $h$. 
\end{lemma}

\begin{proof}
It is obvious since $\omega(1) = L_0$. 
\end{proof}

\begin{define}
Let $M$ be a $V$-module. 
We define two bilinear operations $a \ast v$ and $v \ast a$ 
for homogeneous $a \in V$, $v\in M$ as follows: 
\begin{eqnarray}
a \ast v & = & \text{Res}_z \left( Y(a, z) \frac{(z+1)^{\text{deg} a}}{z} v \right), \nonumber \\ 
v \ast a  & = & \text{Res}_z \left( Y(a, z) \frac{(z+1)^{\text{deg} a-1}}{z} v \right). \nonumber
\end{eqnarray}
Define $A(M) = M/O(M)$ 
where $O(M)$ is the linear span of elements of type 
$$
\text{Res}_z \left( Y(a, z) \frac{(z+1)^{\text{deg} a}}{z^2} v \right). 
$$
We denote the image of $v \in M$ in $A(M)$ by $[v]$. 
\end{define}

\begin{lemma}
The two operations defined above induce an $A(V)$-bimodule structure on $A(M)$. 
\end{lemma}

\begin{lemma} \label{2}
If $M$ is a module of $V$ and $M'$ is a submodule of $M$, 
then the bimodule $A(M/M')$ associated to the quotient module $M/M'$
is isomorphic to $A(M)/[M']$ 
where $[M']$ is the image of $M'$ under the projection from $M$ to $A(M)$. 
\end{lemma}

\begin{lemma}\label{3}
As an $A(M_c)$-bimodule, 
$A(M(h, c))$ is isomorphic to $\mathbb C[x, y]$ 
where $A(M_c) \cong \mathbb C[x] \cong \mathbb C[y]$ acts as multiplications by polynomials in $x$ and $y$. 
\end{lemma}

\begin{proof}
By definition, 
for any $v \in M(h, c)$ we have 
$$
[\omega] \cdot [v] = [ (L_0 + 2 L_{-1} + L_{-2}) v], \qquad 
[v] \cdot [\omega] = [ (L_{-2} + L_{-1}) v]. 
$$
Moreover, 
it can be shown that $O(M(h, c))$ is spanned by elements 
$$
(L_{-n-3} + 2 L_{-n-2} + L_{-n-1}) v 
= \text{Res}_z \left( Y(\omega, z) \frac{(z+1)^2}{z^{2+n}} v \right)
$$
where $n \geq 0$, $v \in M(h, c)$ (see [W, Lemma 4.1]). 
This implies that the map 
$$
\varphi: \mathbb C[x, y] \to A(M(h, c)); \qquad  x^m y^n \mapsto [\omega]^m \cdot [v_{h, c}] \cdot [\omega]^n
$$
is surjective. 
In fact, 
it is also injective (see [DMZ, Proposition 3.1], [L, Proposition 2.15]).  
\end{proof}

\begin{lemma} \label{poly}
$A(L(\triangle(n), c))$ is isomorphic to $\mathbb C[x, y]/(f_n (x, y))$, 
where the polynomial $f_n(x, y)$ is given by 
$$
\sum_{n+1 = p_1 + \cdots + p_r, p_i \geq 1} 
\frac{n!^2}{\prod_{i=1}^{r-1} (p_1 + \cdots p_i)( p_{i+1} + \cdots p_r)} \varkappa^{-n - 1 + r}
$$
$$
(x - p_1 y - p_2 - \cdots - p_r - \triangle(n)) \cdots
(x - p_r y - \triangle(n)). 
$$
\end{lemma}

\begin{proof}
By Lemma \ref{2} and \ref{3}, 
to determine $A(L(\triangle(n), c))$
it suffices to compute the polynomial 
which corresponds to the singular vector $S_{1, n+1} v_{\triangle(n), c}$ in Lemma \ref{sing}. 
Since 
$$
[(L_{-n-3} + 2L_{-n-2} + L_{-n-1}) v] =0$$ 
for $n>0, v \in M(\triangle(n), c)$,  
it can be shown by induction that 
$$
[L_{-n} v] = (-1)^n [ ((n-1) L_{-2} + (n-2) L_{-1}) v]
$$ 
for $n \geq 1$. 
If $[v]$ is mapped to $g_v (x, y)$ 
under the isomorphism 
$\varphi: A(M(\triangle(n), c)) \to \mathbb C[x, y]$ in Lemma \ref{3}, 
then 
$$
\varphi ( [(L_{-1} + L_0)v ] ) = (x -y) g_v, \qquad 
\varphi( [(L_{-2} - L_0) v ] ) =  (2y-x) g_v. 
$$
If furthermore 
$L_0 \cdot v = h \, v$, 
then 
$$
[L_{-n} v]  = (-1)^n [((n-1)(L_{-2} -L_0) + (n-2) (L_{-1} + L_0) + L_0 )v ] 
$$
is mapped to 
$$
(-1)^n ( (n-1)(2y-x) + (n-2) (x-y) + h ) g_v
= (-1)^n ( -x + ny + h) g_v
$$
under the isomorphism $\varphi$. 
It follows that 
the polynomial corresponding to $S_{1, n+1} v_{\triangle(n), c}$ 
is 
\begin{eqnarray}
& & \sum_{n+1 = p_1 + \cdots + p_r, p_i \geq 1} 
\frac{n!^2}{\prod_{i=1}^{r-1} (p_1 + \cdots p_i)( p_{i+1} + \cdots p_r)} (-\varkappa^{-1})^{n + 1 - r}
\nonumber \\
& & (-1)^{p_1} (-x + p_1 y + p_2 + \cdots + p_r + \triangle(n)) 
\cdots
(-1)^{p_r} (-x + p_r y + \triangle(n)) 
\nonumber 
\end{eqnarray}
\end{proof}

\begin{examples}
We list the polynomials $f_n (x, y)$ for small values of $n$: 
\begin{eqnarray}
f_0 & = & x-y; \hspace{.1in}
\text{in this case $\mathbb C[x, y]/(f_0) \cong \mathbb C[x]$, we recover the algebra $A(L(0, c))$.}\nonumber \\
f_1 & = & (x-y-\triangle(1)-1)(x-y-\triangle(1)) + \varkappa^{-1} (x- 2y- \triangle(1))  \nonumber \\
& = & (x - y - \triangle(1)) (x - y - \triangle(-1)) - \frac{1}{\varkappa} y;  \nonumber \\
f_2 & = & (x -y - \triangle(2) -2 ) (x - y -  \triangle(2) -1) ( x -y -  \triangle(2)) \nonumber \\
&  & + \quad 2 \varkappa^{-1} (x - y -  \triangle(2) -2) (x -2y -  \triangle(2)) \nonumber \\
& & + \quad 2 \varkappa^{-1} (x - 2y -  \triangle(2) -1) ( x - y -  \triangle(2))  \nonumber \\
& & + \quad 4 \varkappa^{-2} (x - 3y -  \triangle(2)) \nonumber \\
& = & (x -y ) [(x -y -  \triangle(2)) (x - y -  \triangle(-2)) - \frac{4}{\varkappa} y]. \nonumber 
\end{eqnarray}
\end{examples}

In fact, 
all $f_n$ can be decomposed into products of linear and quadratic polynomials.  

\begin{define}
We define
$g_0 (x, y) = x-y$ and  
\begin{eqnarray} 
g_s (x, y) & = & (x - y)^2 - (\triangle(s) + \triangle(-s)) (x+y) + \triangle(s) \triangle(-s)  \nonumber \\
& =  & (x - y - \triangle(s)) (x - y - \triangle(-s)) - \frac{s^2}{\varkappa} y \nonumber 
\end{eqnarray}
for $s \in \mathbb N^+$ 
(note that $\triangle(s) + \triangle(-s) = \frac{s^2}{2\varkappa}$). 
\end{define} 

\begin{lemma} \label{matrix}
For $k \in \mathbb N$, 
we have 
$f_{2k} = g_0 g_2 \cdots g_{2k}$ and 
$f_{2k+1} = g_1 g_3 \cdots g_{2k+1}$. 
\end{lemma}

\begin{proof}
The polynomial $f_n$ can be realized as the determinant of the following matrix $A_n$: 
$$
\left( 
\begin{array}{cccccc} 
t -n & t - y -n+1 & t - 2y - n+ 2 & \ddots & \ddots & t- ny \\
-\frac{1\cdot n}{\varkappa} & t - n+1 & t- y  -n +2 & t - 2y  - n +3 & \ddots & t - (n-1)y  \\
0 & - \frac{ 2(n-1) }{\varkappa} & t  - n+2 & \ddots & \ddots & t - (n-2)y  \\
0 &  0 & \ddots & \ddots & \ddots & \ddots \\
0 & 0 & 0 & \ddots & t  -1 & t - y  \\
0 & 0 & 0 & 0 & - \frac{n \cdot 1}{\varkappa} & t \\
\end{array}
\right)
$$
where $t = x -y - \triangle(n)$. 
Subtracting the second row from the first, then the third row from the second, and so on, finally the $(n+1)$-th row from the $n$-th, we obtain the following matrix $A_n'$ with the same determinant
$$
\left( 
\begin{array}{cccccc} 
t -n + \frac{1\cdot n}{\varkappa} & -y & -y & \ddots & \ddots & -y \\
-\frac{1\cdot n}{\varkappa} & t  - n+1 + \frac{ 2(n-1) }{\varkappa} & - y  & -y  & \ddots & -y \\
0 & - \frac{ 2(n-1) }{\varkappa} & t - n+2 + \frac{ 3(n-2) }{\varkappa}  & \ddots & \ddots & -y \\
0 &  0 & \ddots & \ddots & \ddots & \vdots \\
0 & 0 & 0 & \ddots & t -1+ \frac{ n \cdot 1 }{\varkappa} & - y \\
0 & 0 & 0 & 0 & - \frac{n \cdot 1}{\varkappa} & t  \\
\end{array}
\right). 
$$
Let $P_n = (p_{i, j})_{1 \leq i, j \leq n+1}$ be the following matrix:
$$
\left(
\begin{array}{cccccc}
n - \frac{n}{\varkappa} + (n-1)y & y & y & \cdots & \cdots & y \\
n-1  - \frac{n-2}{\varkappa} + (n-2) y & \frac{2}{\varkappa}-1 & y & \ddots & \ddots & y \\
n-2 - \frac{n-4}{\varkappa} + (n-3) y & - \frac{2}{\varkappa} & \frac{6}{\varkappa} -2 & y & \cdots &  y \\
\vdots & \ddots & \ddots & \ddots & \ddots & \vdots \\
1+ \frac{n-2}{\varkappa} &  \cdots & \ddots & -\frac{(n-2)(n-1)}{\varkappa} & \frac{(n-1)n}{\varkappa} - n +1 & y \\
\frac{n}{\varkappa} & \cdots & \cdots & 0 & -\frac{(n-1)n}{\varkappa} & 0
\end{array}
\right)
$$
i.e.
$p_{i, 1} = n - i+1 - \frac{ n -2i+2}{ \varkappa} + (n-i) y$ for $1\leq i \leq n-1$; 
$p_{n, 1} = 1 + \frac{n-2}{\varkappa}$; 
$p_{n+1, 1} = \frac{n}{\varkappa}$; 
$p_{i, j} = y$ for $1 \leq i < j \leq n+1$; 
$p_{i, i} = \frac{(i-1)i}{\varkappa} - i+1$ for $2 \leq i \leq n$; 
$p_{n+1, n+1} =0$; 
$p_{i+1, i} = - \frac{(i-1)i}{\varkappa}$ for $2 \leq i \leq n$; 
$p_{i, j} =0$ for $2 \leq j \leq n-1, i > j+1$. 
Then 
$$
A_n' P_n = P_n A_n''
$$
where $A_n''$ is the matrix
$$
\left( 
\begin{array}{cccccc}
t - n + \frac{n}{\varkappa} &  0 & \cdots & \cdots & 0 & -y \\
\ast & t - n +1 + \frac{2(n-1)}{\varkappa} & -y & \cdots & -y & 0 \\
\ast & - \frac{n-2}{\varkappa} & t - n + 2 + \frac{3(n-2)}{\varkappa} & \ddots & -y & 0 \\
\vdots & \ddots & \cdots & \ddots & \ddots & \vdots \\
\ast & 0 & \cdots & - \frac{n-2}{\varkappa} & t-1+ \frac{n}{\varkappa} & 0 \\
-\frac{n^2}{\varkappa} & 0 & \cdots & 0 & 0 & t 
\end{array}
\right). 
$$
Note that 
$$
\triangle(n) - \triangle(n-2i) = \frac{i(n+1-i)}{\varkappa} - i, 
$$
hence for $1\leq i \leq n-1$, 
we have
\begin{eqnarray}
A_n'' [i+1, i+1] 
& = &  t - i + \frac{i(n+1-i)}{\varkappa} 
= x - y - \triangle(n)  - i + \frac{i(n+1-i)}{\varkappa} 
\nonumber \\
& = & x - y - \triangle(n-2i)
= x - y - \triangle(n-2) - i+1 + \frac{(i-1)(n-i)}{\varkappa}. 
\nonumber 
\end{eqnarray}
This means that 
$$
A_n''[2, \cdots, n; 2, \cdots, n] = A_{n-2}', 
$$
where $A_n''[2, \cdots, n; 2, \cdots, n]$ is the submatrix of $A_n''$ 
obtained by deleting the first and $n+1$-th rows and columns. 
Now, 
$$
\text{det} A_n' = \text{det} A_n'' 
= \text{det} A_{n-2}' \text{det} 
\left(
\begin{array}{cc}
t-n+\frac{n}{\varkappa} & -y \\
-\frac{n^2}{\varkappa} & t
\end{array}
\right). 
$$
Since $t = x - y - \triangle(n)$, 
we have 
$t - n + \frac{n}{\varkappa} = x - y - \triangle(-n)$, 
hence
$$
\text{det} 
\left(
\begin{array}{cc}
t-n+\frac{n}{\varkappa} & -y \\
-\frac{n^2}{\varkappa} & t
\end{array}
\right)
= g_n (x, y). 
$$
We obtained a recursion 
$$
\text{det} A_n' = g_n \, \text{det} A_{n-2}'. 
$$
Obviously
$\text{det} A_0' = g_0$ and $\text{det} A_1' = g_1$, 
the lemma is now clear. 
\end{proof}


\begin{lemma}
Suppose $M^i = \oplus_{n=0}^\infty M^i_n (i = 1, 2, 3)$ are irreducible representations of a VOA $V$ 
and $L_0|_{M^i_n} =(h_i +n) \text{Id}$. 
An intertwining operator $I(\cdot, z)$ of type 
$\left( \begin{array}{ccc} \quad & M^{3} & \quad \\  M^{1} & \quad & M^{2} \end{array} \right)$ 
can be written as 
$$
I(v, z) = \sum_{n \in \mathbb Z} v(n) z^{-n-1} z^{-h_1-h_2+h_3}
$$ 
so that for homogeneous $v \in M^1$, 
$$
v(n) M^2_m \subset M^3_{m-n-1+\text{deg} v}
$$
where 
$\text{deg} v = k$ means that $v \in M^1_k$. 
In particular, 
we have 
$v (\text{deg} \, v -1) M^2_0 \subset M^3_0$.
The map 
$$
A(M^1) \otimes_{A(V)} M^2_0 \to M^3_0; \qquad [v] \otimes m \mapsto v(\text{deg} v-1) m
$$
is a map of $A(V)$-modules. 
\end{lemma}

We denote by $I_{M^1 M^2}^{M^3}$ the space of intertwining operators. 
The above lemma implies that there is a linear map 
$$
\pi: I_{M^1 M^2}^{M^3} \to \text{Hom}_{A(V)} (A(M^1) \otimes_{A(V)} M^2_0, M^3_0). 
$$


\begin{lemma}{[FZ]} \label{fzf}
If the vertex operator algebra $V$ is rational, 
then the map $\pi$ is an isomorphism. 
\end{lemma}

\begin{remark}
If $V$ is not rational, 
it is proved in [L, Theorem 2.11] that 
$\pi$ is an isomorphism if we require that $M^1, M^2, M^3$ are lowest weight $V$-modules 
such that $M^2$ and the contragredient module of $M^3$ are generalized Verma $V$-modules. 
For definitions of lowest weight modules and generalized Verma modules, see [L]. 
Under these assumptions, 
the module $M^3$ is irreducible. 
Examination of the proof shows that 
the theorem still holds 
if we forsake the lowest-weight assumption on $M^3$, 
instead only requiring that $M^3$ has a generalized Verma $V$-module as its dual; 
this way, $M^3$ does not have to be irreducible either. 
In the Virasoro case, 
generalized Verma modules coincide with the Verma modules we defined. 
\end{remark}

\begin{define}\label{contra}
Let $M = \bigoplus_{h\in \mathbb C} M_h$ be a $V$-module. 
Let $M^* = \bigoplus_h M_h^*$ and define
$$
\langle Y^*(a, z) \, f,  \,v  \, \rangle
= \langle  \,f  \,,  Y(e^{z L_1} (-z^{-2})^{L_0} a, z^{-1})  \, v \rangle
$$
for $a \in V, f \in M^*, v \in M$. 
Then $(M^*, Y^*)$ carries the structure of a $V$-module. 
We call $M^*$ the contragredient module of $M$. 
\end{define}

We have the following generalization of Lemma \ref{fzf} and [L, Theorem 2.11]. 

\begin{lemma}\label{generalized}
If $M^1$ is a lowest weight $V$-module, 
$M^2$ and ${M^3}^*$ (the contragredient module of $M^3$) 
are generalized Verma $V$-modules, 
then the linear map $\pi$ is an isomorphism. 
\end{lemma}

The Verma module $M(h, c)$ has the following $\mathbb N$-grading: 
$$
\text{deg } L_{-i_1} L_{-i_2} \cdots L_{-i_n} v_{h, c} = i_1 + \cdots i_n. 
$$
There is a unique symmetric bilinear form $(, )$ on $M(h, c)$ which respects the above grading and satisfies 
$$
(v_{h, c}, v_{h, c}) = 1,  \qquad
(L_n v, v') = (v, L_{-n} v'), 
$$ 
for $n \in \mathbb Z, v, v' \in M(h, c)$. 
The kernel of this bilinear form is $J(h, c)$. 
Let $M(h, c)^*$ be the contragredient module of $M(h, c)$ as defined in \ref{contra}. 
The Virasoro algebra acts on $M(h, c)^*$ as follows: 
$$
\langle L_n \, f, v \rangle  = \langle f, \, L_{-n} \, v \rangle, 
\qquad \underline c \cdot f = c \, f, 
$$
for $f \in M(h, c)^*, v \in M(h, c)$. 
The bilinear form on $M(h, c)$
induces a map of Virasoro modules 
$$
\tau: M(h, c) \to M(h, c)^*
$$
with kernel $J(h, c)$. 
Its image, 
isomorphic to $L(h, c)$, 
is the unique irreducible submodule of $M(h, c)^*$. 

\begin{lemma}
For any $h_i \in \mathbb C, i = 1, 2, 3$, $c \in \mathbb C$,  
the space of intertwining operators 
$
I_{M(h_1, c) M(h_2, c)}^{M(h_3, c)^*}
$
has dimension $1$.
\end{lemma}

\begin{proof}
It follows from Lemma \ref{generalized} that the map $\pi$ is an isomorphism. 
Previous lemmas show that 
$$
A(M_c) \cong \mathbb C[x] \cong \mathbb C[y], 
\qquad 
A(M(h_1, c)) \cong \mathbb C[x, y]. 
$$
It is obvious that 
$$
\text{dim} \,\text{Hom}_{\mathbb C[x]} 
(\mathbb C[x, y] \otimes_{\mathbb C[y]} \mathbb C_{h_2},  \mathbb C_{h_3}) = 1. 
$$
Here the tensor product over $\mathbb C[y]$ is obtained by specializing $y$ to $h_2$, 
and $\mathbb C[x]$ acts on $\mathbb C_{h_3}$ by specializing $x$ to $h_3$.  
\end{proof} 

For each permutation $\{i, j, k\}$ of $\{ 1, 2, 3\}$, 
choose an intertwining operator 
$$
I_{i, j}^k(\cdot, z): M(h_i, c) \otimes M(h_j, c) \to M(h_k, c)^* \{ z \}
$$
such that 
$\langle v_{h_k, c}, v_{h_i, c}(-1) v_{h_j, c} \rangle =1$ and 
define a map 
$$
\psi_{i, j}^k: A(M(h_i, c)) \cong \mathbb C[x, y] \to \mathbb C;  \quad g(x, y) \mapsto g(h_k, h_j). 
$$
Note that the exponents of $z$ appearing in $I_{i, j}^k(\cdot, z)$ 
are all equal to $(h_k - h_i -h_j)$ ( mod $\mathbb Z$). 


\begin{lemma} \label{coeff}
Let $m_i \in M(h_i, c)$, $i =1, 2, 3$, be homogeneous vectors. 
Let $\rho = h_3 - h_1 - h_2$, 
then 
\begin{eqnarray}
\langle I_{1, 2}^3 (m_1, z) v_{h_2, c}, v_{h_3, c} \rangle 
& = & z^{\rho} z^{- \text{deg}\, m_1} \psi_{1, 2}^3 ([m_1]),  
\\
\langle I_{1, 2}^3 (v_{h_1, c}, z) m_2, v_{h_3, c} \rangle 
& = & z^{\rho} (-z)^{ -  \text{deg} \, m_2} \psi_{2, 1}^3 ([m_2]),
\\
\langle I_{1, 2}^3 (v_{h_1, c}, z) v_{h_2, c}, m_3 \rangle
& = & z^{\rho}  (-z)^{ \text{deg} \, m_3} \psi_{3, 1}^2 ([m_3]). 
\end{eqnarray}
\end{lemma}

\begin{proof}
(2.1) is clear.
Let 
$$
\Omega_r: I_{M(h_i, c) M(h_j, c)}^{M(h_k, c)^*} \to I_{M(h_j, c) M(h_i, c)}^{M(h_k, c)^*}, 
\qquad 
\mathcal A_s: I_{M(h_i, c) M(h_j, c)}^{M(h_k, c)^*} \to I_{M(h_i, c) M(h_k, c)}^{M(h_j, c)^*}
$$
be the isomorphisms defined in [HL].  
It is easy to check that 
$$
\Omega_r (I_{i, j}^k) = e^{(2r+1) \pi i [h_k - h_i - h_j] } I_{j, i}^k, 
\qquad 
\mathcal A_s (I_{i, j}^k) = e^{ (2s+1) \pi i h_i } I_{i, k}^j. 
$$
By the definition of $\Omega_r$, 
we have 
\begin{eqnarray}
& & 
\langle \Omega_r (I_{1, 2}^3) (m_2, z) v_{h_1, c}, v_{h_3, c} \rangle
= \langle e^{z L_{-1} } I_{1, 2}^3 (v_{h_1, c}, e^{(2r+1) \pi i} z) m_2, v_{h_3, c} \rangle
\nonumber \\
& = & 
\langle I_{1, 2}^3 (v_{h_1, c}, e^{(2r+1) \pi i} z) m_2, v_{h_3, c} \rangle
= e^{(2r+1) \pi i [\rho - \text{deg} \, m_2] } (I_{1, 2}^3 (v_{h_1, c}, z) m_2, v_{h_3, c} \rangle. 
\nonumber 
\end{eqnarray}
Since 
$$
\Omega_r (I_{1, 2}^3) = e^{(2r+1) \pi i \rho} I_{2, 1}^3
$$
and 
$$
\langle I_{2, 1}^3 (m_2, z) v_{h_1, c}, v_{h_3, c} \rangle 
= z^{\rho - \text{deg} \, m_2} \psi_{2, 1}^3 ([m_2]), 
$$
(2.2) follows. 
(2.3) can be proved similarly using $\Omega_r$ and $\mathcal A_s$. 
\end{proof}


\begin{lemma} \label{descend}
The fusion rule 
$\text{dim} \,\, I_{L(h_1, c) L(h_2, c)}^{L(h_3, c)} \leq 1$ and the equality holds if and only if the maps 
$\psi_{1, 2}^3$, $\psi_{2, 1}^3$, and $\psi_{3, 1}^2$ kill the subspaces $[J(h_1, c)]$, $[J(h_2, c)]$, and $[J(h_3, c)]$ respectively. 
\end{lemma}

\begin{proof}
Since $L(h_i, c)$ is a quotient of $M(h_i, c)$ and a submodule of $M(h_i, c)^*$, 
it is clear that 
$$
\text{dim} I_{L(h_1, c) L(h_2, c)}^{L(h_3, c)} 
\leq \text{dim} I_{M(h_1, c) M(h_2, c)}^{M(h_3, c)^*} 
=1. 
$$
The equality holds if and only if $I_{1, 2}^3$ descends to an intertwining operator of the previous type. 
First, the image of $I_{1, 2}^3(\cdot, z)$ lying in $L(h_3, c)$ is equivalent to 
\begin{eqnarray}
\langle I_{1, 2}^3( m_1, z) m_2, m_3 \rangle =0 
\end{eqnarray}
for all $m_3 \in J(h_3, c)$ and $m_i \in M(h_i, c)$, $i=1, 2$. 
Since $M(h_1, c)$, $M(h_2, c)$ are generated by the lowest weight subspaces and $J(h_3, c)$ is a submodule of $M(h_3, c)$, 
(2.4) is equivalent to 
\begin{eqnarray}
\langle I_{1, 2}^3( v_{h_1, c}, z) v_{h_2, c}, m_3 \rangle =0 
\end{eqnarray}
for all $m_3 \in J(h_3, c)$. 
(2.3) implies that (2.5) holds 
iff the image of $J(h_3, c)$ in $A(M(h_3, c))$ is killed by the map $\psi_{1, 2}^3$. 
Similarly the intertwining operator 
$I_{1, 2}^3$ descends to $L(h_1, c)$ and $L(h_2, c)$ on the other two factors if and only if
$$
\langle I_{1, 2}^3(m_1, z) v_{h_2, c}, v_{h_3, c} \rangle 
= \langle I_{1, 2}^3(v_{h_1, c}, z) m_2, v_{h_3, c} \rangle =0
$$
for all $m_i \in J(h_i, c)$, $i = 1, 2$. 
Again by Lemma \ref{coeff}, 
this is equivalent to the images of $J(h_i, c)$ in $A(M(h_i, c))$, $i =1, 2$, 
being killed by the corresponding $\psi$-maps. 
\end{proof}

\begin{lemma} \label{bb1}
For any $n \in \mathbb Z$ and $s \in \mathbb N$, 
we have 
$$
\triangle (n+s) + \triangle(n -s) = 2 \triangle(n) + \triangle(s) + \triangle(-s), 
$$
$$
\triangle (n+s)  \triangle(n -s) = ( \triangle(n) - \triangle(s))  ( \triangle(n) -  \triangle(-s)). 
$$
\end{lemma}

\begin{proof}
This is pure calculation. 
\end{proof}

\begin{prop}
For $k_i \in \mathbb N, i = 1, 2, 3$, we have 
the fusion rule 
$$
\text{dim} \, I_{L(\triangle(k_1), c) L(\triangle(k_2), c)}^{L(\triangle(k_3), c)} \leq 1. 
$$
The equality holds if and only if 
$k_1 + k_2 + k_3 \equiv 0 \text{ (mod } 2 \mathbb Z)$ and 
$| k_1 - k_2 | \leq k_3 \leq k_1 + k_2$. 
\end{prop}

\begin{proof}
By Lemma \ref{descend} and Lemma \ref{poly}, 
the fusion rule is $1$ if and only if 
\begin{eqnarray}
f_{k_1} (\triangle(k_2), \triangle(k_3)) & = & 0, \\
f_{k_2} (\triangle(k_1), \triangle(k_3)) & = & 0, \\
f_{k_3} (\triangle(k_1), \triangle(k_2)) & = & 0. 
\end{eqnarray}
Recall the polynomials 
$$
g_s (x, y) 
= (x - y)^2 - (\triangle(s) + \triangle(-s)) (x+y) + \triangle(s) \triangle(-s) 
$$
$$
= x^2 - (2y + \triangle(s) + \triangle(-s)) x + (y - \triangle(s)) (y- \triangle(-s)). 
$$
Lemma \ref{bb1} shows that if we set $y = \triangle(n)$ for some $n \in \mathbb Z$, 
then the solutions to the equation $g_s (x, \triangle(n)) = 0$ are
$\triangle(n+s)$ and $\triangle(n-s)$. 
Therefore, 
by Lemma \ref{matrix}, 
(2.6) is true iff $k_2 \in \{ k_3 + k_1, k_3 + k_1 -2, \cdots, k_3 - k_1\}$, 
i.e. 
\begin{eqnarray}
k_1 + k_2 + k_3 \equiv 0 \text{ (mod } 2 \mathbb Z)
\end{eqnarray}
and  
$$
k_3 - k_1 \leq k_2 \leq k_3 + k_1. 
$$
Similarly, 
(2.7) (resp. 2.8) is true iff (2.9) is true and 
$$
k_3 - k_2 \leq k_1 \leq k_3 + k_2 \quad 
(\text{ resp. } k_2 - k_3 \leq k_1 \leq k_2 + k_3). 
$$
The proposition is now clear. 
\end{proof}

\section{BPZ-equations and locality}

The matrix coefficients of the iterates of intertwining operators between irreducible Virasoro modules satisfy the so-called BPZ-equations ([BPZ]). 
These differential equations come from the null vectors. 
The null vector in $L(\triangle(1), c)$ yields a second-order differential equation; 
null vectors of higher conformal weights give higher-order differential equations. 

We first derive the operator version of the second-order equation. 

\begin{lemma}
Let $\Phi(\cdot, z): L(\triangle(1), c) \to \text{End} (M, M')  \{z\}$ be an intertwining operator. 
Then the formal power series $\Phi(v_{\triangle(1), c}, z)$ with coefficients in $\text{End} (M, M')$ satisfies the following equation
$$
\frac{d^2}{d z^2} \Phi(v_{\triangle(1), c}, z) 
= \frac{1}{\varkappa} : L(z) \Phi(v_{\triangle(1), c}, z): .
$$
\end{lemma}

\begin{proof}
The null vector $( L_{-1}^2 - \varkappa^{-1} L_{-2}) v_{\triangle(1), c}$ in $M(\triangle(1), c)$ 
(see Example \ref{null}) 
is killed in the irreducible module $L(\triangle(1), c)$. 
Therefore, 
the intertwining operator associated to it
$$
\Phi(( L_{-1}^2 - \varkappa^{-1} L_{-2}) v_{\triangle(1), c}, z)
$$ 
is zero. 
The properties of intertwining operators imply that 
\begin{eqnarray}
\Phi ( L_{-1}^2 v_{\triangle(1), c}, z) & = & 
\frac{d^2}{d z^2} \Phi(v_{\triangle(1), c}, z), 
\nonumber \\
\Phi ( L_{-2} v_{\triangle(1), c}, z) & = & 
: L(z) \Phi(v_{\triangle(1), c}, z): .
\nonumber
\end{eqnarray}
The lemma is now clear. 
\end{proof}

For any $n\geq 0$, 
denote by $\Phi^+(\cdot, z)$ and $\Phi^-(\cdot, z)$ the intertwining operators 
$$
\Phi^+(\cdot, z): L(\triangle(1), c) \otimes L( \triangle(n), c) \to L( \triangle(n+1), c) \{z\}
$$
$$
\Phi^-(\cdot, z): L(\triangle(1), c) \otimes L( \triangle(n), c) \to L( \triangle(n-1), c) \{z\}
$$ 
such that 
$$
(v_{\triangle(n+1), c}^*,  \Phi^+(v_{\triangle(1), c}, z) v_{\triangle(n), c} )
= z^{\triangle(n+1) - \triangle(n) - \triangle(1)} \cdot 1
$$
$$
(v_{\triangle(n-1), c}^*,  \Phi^-(v_{\triangle(1), c}, z) v_{\triangle(n), c} )
= z^{\triangle(n-1) - \triangle(n) - \triangle(1)} \cdot 1. 
$$
If $n=0$, only $\Phi^+$ exists. 
Consider the following matrix coefficients of iterates of intertwining operators 
$$
\Phi(z, w)= (v_{\triangle(m), c}^*, \Phi^{\pm} (v_{\triangle(1), c}, z) \Phi^{\pm} (v_{\triangle(1), c}, w) v_{\triangle(n), c})
$$
whenever it is well defined. 

\begin{prop} \label{BPZ}
$\Phi(z, w)$ satisfies the following differential equations: 
\begin{eqnarray}
\partial_z^2 \Phi(z, w)
& = & 
\frac{1}{\varkappa} 
\left( \frac{w}{(z-w)z} \partial_w 
-\frac{1}{z} \partial_z + \frac{\triangle(n)}{z^2} + \frac{\triangle(1)}{(z-w)^2} \right) 
\Phi
\\
\partial_w^2 \Phi(z, w)
& = & 
\frac{1}{\varkappa} 
\left( \frac{z}{(w-z)w} \partial_z 
-\frac{1}{w} \partial_w + \frac{\triangle(n)}{w^2} + \frac{\triangle(1)}{(z-w)^2} \right) 
\Phi
\end{eqnarray}
where the functions on the right-hand side are understood to be power series expansions in the domain 
$|z| > |w| >0$. 
\end{prop}

\begin{proof}
By properties of intertwining operators,
we have 
$$
[L_{-1}, \Phi^{\pm} (v_{\triangle(1), c}, z) ] 
= \partial_z \Phi^{\pm} (v_{\triangle(1), c}, z), 
$$
$$
[L^+(z), \Phi^{\pm} (v_{\triangle(1), c}, w)] 
= \frac{ 
\partial_w \Phi^{\pm} (v_{\triangle(1), c}, w)  }
{z-w} 
+ \frac{ 
\triangle(1) \Phi^{\pm} (v_{\triangle(1), c}, w) }
{(z-w)^2}, 
$$
where
$$
L^+ (z) = \sum_{n \geq  -1} L_n z^{-n-2}, \qquad
L^- (z)  = \sum_{n <  -1} L_n z^{-n-2}. 
$$ 
Now, 
\begin{eqnarray} 
& &  \partial^2_z \Phi(z, w)
\nonumber \\ 
&=& 
(v_{\triangle(m), c}^*, \partial_z^2 \Phi^{\pm} (v_{\triangle(1), c}, z) \Phi^{\pm} (v_{\triangle(1), c}, w) v_{\triangle(n), c})
\nonumber \\
&=&  
\frac{1}{\varkappa} (v_{\triangle(m), c}^*, :L(z) \Phi^{\pm} (v_{\triangle(1), c}, z): \Phi^{\pm} (v_{\triangle(1), c}, w) v_{\triangle(n), c})
\nonumber \\
&=&  
\frac{1}{\varkappa} (v_{\triangle(m), c}^*, 
[ L^-(z) \Phi^{\pm} (v_{\triangle(1), c}, z) + \Phi^{\pm} (v_{\triangle(1), c}, z) L^+ (z) ]
 \Phi^{\pm} (v_{\triangle(1), c}, w) v_{\triangle(n), c})
\nonumber \\
&=& 
\frac{1}{\varkappa} (v_{\triangle(m), c}^*, 
\Phi^{\pm} (v_{\triangle(1), c}, z) L^+ (z) 
\Phi^{\pm} (v_{\triangle(1), c}, w) v_{\triangle(n), c})
\nonumber \\
& =&
\frac{1}{\varkappa} (v_{\triangle(m), c}^*, 
\Phi^{\pm} (v_{\triangle(1), c}, z) 
[L^+ (z), \Phi^{\pm} (v_{\triangle(1), c}, w)] 
v_{\triangle(n), c})
\nonumber \\ 
& & 
+ \quad \frac{1}{\varkappa} (v_{\triangle(m), c}^*, 
\Phi^{\pm} (v_{\triangle(1), c}, z) \Phi^{\pm} (v_{\triangle(1), c}, w) L^+(z) v_{\triangle(n), c})
\nonumber \\
& = & 
\frac{1}{\varkappa} (v_{\triangle(m), c}^*, 
\Phi^{\pm} (v_{\triangle(1), c}, z) 
\left[ \frac{  \partial_w \Phi^{\pm} (v_{\triangle(1), c}, w)  } {z-w} 
+ \frac{ \triangle(1) \Phi^{\pm} (v_{\triangle(1), c}, w) } {(z-w)^2} \right]
v_{\triangle(n), c})
\nonumber \\
& & 
+ \quad \frac{1}{\varkappa} (v_{\triangle(m), c}^*, 
\Phi^{\pm} (v_{\triangle(1), c}, z) \Phi^{\pm} (v_{\triangle(1), c}, w) 
[L_{-1} z^{-1} + L_0 z^{-2} ] 
v_{\triangle(n), c})
\nonumber \\
& = &
\frac{1}{\varkappa} \left( \frac{\partial_w}{z-w}  + \frac{\triangle(1)} {(z-w)^2} \right) \Phi(z, w)
+ \frac{1}{\varkappa} \left( \frac{ - \partial_z - \partial_w}{z} + \frac{\triangle(n)}{z^2} \right) \Phi(z, w)
\nonumber \\
&=&
\frac{1}{\varkappa} 
\left( \frac{w}{(z-w)z} \partial_w 
-\frac{1}{z} \partial_z + \frac{\triangle(n)}{z^2} + \frac{\triangle(1)}{(z-w)^2} \right) 
\Phi(z, w) 
\nonumber 
\end{eqnarray} 
The second equation is derived analogously. 
\end{proof}

Since 
$$
\triangle(n+1) - \triangle (n) - \triangle(1) = \frac{n}{2 \varkappa}, 
\qquad 
\triangle(n-1) - \triangle (n) - \triangle (1) = 1 - \frac{n+2}{2\varkappa}, 
$$
we look for solutions of the form
\begin{eqnarray}
\Phi (z, w) & = & z^{\alpha} w^{\beta} f(t)
\end{eqnarray}
where 
$f(t) \in \mathbb C[[t]]$ with $t = w/z$, $f(0)=1$, 
and $(\alpha, \beta)$ takes the following values
\begin{eqnarray}
( \frac{n+1}{2 \varkappa}, \frac{n}{2 \varkappa}),  
& &
(1 - \frac{n+3}{2 \varkappa}, \frac{n}{2 \varkappa}), 
\\
( \frac{n-1}{2 \varkappa}, 1 - \frac{n+2}{2 \varkappa}), 
& & 
( 1- \frac{n+1}{2 \varkappa}, 1- \frac{n+2}{2 \varkappa}). 
\end{eqnarray}
Plugging (3.3) into (3.1), 
we obtain 
$$
t f''(t) + \left[ 2 (1-\alpha) - \frac{1}{\varkappa} \left(1+ \frac{1}{1-t} \right) \right] f'(t) 
- \frac{1}{\varkappa} \left [ \frac{\beta}{1-t} + \frac{(2-t) \triangle(1)}{(1-t)^2} \right] f(t) = 0. 
$$
Let 
$$f(t) = (1-t)^{\frac{1}{2 \varkappa}} g(t), 
$$
then $g$ satisfies 
\begin{eqnarray}
 t(1-t) g''(t) + [ 2(1-\alpha - \frac{1}{\varkappa} ) - 2(1 - \alpha) t ] g'(t) 
+  \frac{1}{\varkappa} (\alpha - \beta - \frac{1}{2\varkappa} )  g(t) & = &  0. 
\end{eqnarray}
When $\alpha, \beta$ take the specified values in (3.4-3.5) from left to right, up and down, 
(3.6) is reduced to the following:  
\begin{eqnarray} 
t(1-t) g''(t) + \left[ \left(2 - \frac{n+3}{\varkappa} \right) -  \left( 2 - \frac{n+1}{\varkappa} \right) t \right] g'(t) 
& = & 0, 
\\
%
 t(1-t) g''(t) + \left[ \frac{n+1}{\varkappa} -  \frac{n+3}{\varkappa} t  \right] g'(t) 
+ \frac{1}{\varkappa} \left[ 1 - \frac{n+2}{\varkappa} \right] g(t) 
& = & 0, 
\\
%
t(1-t) g''(t) + \left[ \left( 2- \frac{n+1}{\varkappa} \right) - \left(2 - \frac{n-1}{ \varkappa} \right) t \right]  g'(t) 
+ \frac{1}{\varkappa} \left[ \frac{n}{\varkappa} -1 \right] g(t) 
& = & 0, 
\\
%
t(1-t) g''(t) + \left[ \frac{n-1}{\varkappa} - \frac{n+1}{\varkappa} t \right] g'(t) 
& = & 0, 
\end{eqnarray}
Analogously, 
if we substitute $\Phi(z, w) = z^\alpha w^\beta (1-t)^{\frac{1}{2 \varkappa}} g(t)$ into (3.2), 
we obtain
\begin{eqnarray}
t(1-t) g''(t) + \left[ \frac{n+1}{\varkappa} - \frac{n+3}{\varkappa} t \right] g'(t) 
& = & 0, 
\end{eqnarray}
%
%
%
(3.8), (3.9), and 
\begin{eqnarray}
t(1-t) g''(t) +  \left[ \left( 2- \frac{n+1}{\varkappa} \right) - \left(2 - \frac{n-1}{\varkappa} \right) t \right]  g'(t) 
& = & 0. 
\end{eqnarray}


Recall the hypergeometric series 
$$
{}_2F_1(a, b, c; t) = \sum_{n \geq 0} \frac{(a)_n (b)_n}{ n! (c)_n} t^n
$$
where 
$(a)_n = a(a+1) \ldots (a+n-1)$. 
It is well-known that the Gauss hypergeometric equation
$$
t(1-t) F''(t) + [ c- (a+b+1) t ] F'(t) - ab F =0
$$
has two linearly independent solutions around $0$ 
$$
{}_2F_1(a, b, c; t) 
\qquad \text{and} \qquad 
t^{1-c} {}_2F_1(a-c+1, b-c+1, 2-c; t). 
$$
(3.7) - (3.12) are all hypergeometric equations. 
Now,
we can express the matrix coefficients of the iterates of the intertwining operators 
$\Phi^\pm(\cdot, z)$
in terms of the hypergeometric series. 

\begin{prop}\label{matrixcoefficients}
For $n \geq 2$, 
we have
\begin{eqnarray}
(v_{\triangle(n+2), c}^*,  \Phi^+(v_{\triangle(1), c}, z) \Phi^+ (v_{\triangle(1), c}, w) v_{\triangle(n), c})
& = & 
z^{\frac{n+1}{2 \varkappa}} w^{\frac{n}{2 \varkappa}} (1-t)^{\frac{1}{2\varkappa}}
\nonumber \\
(v_{\triangle(n), c}^*,  \Phi^-(v_{\triangle(1), c}, z) \Phi^+ (v_{\triangle(1), c}, w) v_{\triangle(n), c})
& = & 
z^{1 - \frac{n+3}{2 \varkappa}} w^{\frac{n}{2 \varkappa}} (1-t)^{\frac{1}{2\varkappa}}
{}_2F_1(\frac{1}{\varkappa}, \frac{n+2}{\varkappa} -1, \frac{n+1}{\varkappa}; t)
\nonumber \\
(v_{\triangle(n), c}^*,  \Phi^+(v_{\triangle(1), c}, z) \Phi^- (v_{\triangle(1), c}, w) v_{\triangle(n), c})
& = &
z^{\frac{n-1}{2 \varkappa}} w^{1 - \frac{n+2}{2 \varkappa} } (1-t)^{\frac{1}{2\varkappa}}
{}_2F_1(1- \frac{n}{\varkappa}, \frac{1}{\varkappa}, 2- \frac{n+1}{\varkappa}; t)
\nonumber \\
(v_{\triangle(n-2), c}^*,  \Phi^-(v_{\triangle(1), c}, z) \Phi^- (v_{\triangle(1), c}, w) v_{\triangle(n), c})
& = &
z^{1- \frac{n+1}{2 \varkappa}} w^{ 1- \frac{n+2}{2 \varkappa}} (1-t)^{\frac{1}{2\varkappa}}
\nonumber 
\end{eqnarray}
where $t = w/z$. 

If $n =1$, 
we only have the first three; 
if $n=0$,
we only have the first two: 
\begin{eqnarray}
(v_{\triangle(2), c}^*,  \Phi^+(v_{\triangle(1), c}, z) \Phi^+ (v_{\triangle(1), c}, w) v_{0, c})
& = &
(z-w)^{\frac{1}{2\varkappa}} 
\nonumber \\
(v_{0, c}^*,  \Phi^-(v_{\triangle(1), c}, z) \Phi^+ (v_{\triangle(1), c}, w) v_{0, c})
& =& 
(z-w)^{1 - \frac{3}{2 \varkappa}}. 
\nonumber
\end{eqnarray}
\end{prop}

\begin{proof}
This is clear from the above discussion. 
\end{proof}


Recall that we have set
$$
c = 13 - 6 \varkappa - 6\varkappa^{-1}, \qquad \triangle(n) = \frac{n(n+2)}{4\varkappa} - \frac{n}{2}, 
$$ 
for some $\varkappa \notin \mathbb Q$, $n \geq 0$. 
Now, 
define 
$$
\bar c = 13 + 6 \varkappa + 6 \varkappa^{-1}, \qquad \bar \triangle(n) = - \frac{n(n+2)}{4\varkappa} - \frac{n}{2}, 
$$
and consider the following $\text{Vir}_c \oplus \text{Vir}_{\bar c}$-module:
$$
W = \bigoplus_{n \geq 0} L(\triangle(n), c) \otimes L(\bar \triangle(n), \bar c). 
$$
Our goal is to construct a vertex operator algebra structure on $W$ with rank $26$. 
The main idea of the construction is to define the vertex operators by pairing the intertwining operators from the two copies of the Virasoro actions together 
(see [Z] for the same idea applied to affine Lie algebras). 
We want the $0$-component 
$W(0) = L(0, c) \otimes L(0, \bar c)$ 
be a vertex subalgebra of $W$ 
so that the Virasoro module structure of $W$ agrees with the $W(0)$-module structure of $W$.  
Let us make a simplification of notations:
$$
{\bf v}_n = v_{\triangle(n), c} \otimes v_{\bar \triangle(n), \bar c} \in W(n) = L(\triangle(n), c) \otimes L(\bar \triangle(n), \bar c). 
$$
${\bf v}_0$ will be the vacuum vector of $W$. 
The vertex operators of elements of $W(0)$ are defined in the usual way using 
normal ordering of 
$L(z) = \sum_n L_n z^{-n-2}$, 
$\bar L(z) = \sum_n \bar L_n z^{-n-2}$, 
and their derivatives. 
We use bar to denote the second copy of the Virasoro action with the central charge $\bar c$. 
The Virasoro element is given by $(L_{-2} + \bar L_{-2}) {\bf v}_0$. 

We define the vertex operators for elements of $W(1)$. 
For any $v \otimes v' \in W(1)$, $u \otimes u' \in W(n)$, 
we set 
\begin{eqnarray}
\hspace{.6 in} Y(v \otimes v', z) (u \otimes u') 
& =  &
\Phi^+(v, z) u \otimes \Psi^+(v', z) u' + X_{1, n}^{n-1} 
\Phi^-(v, z) u \otimes \Psi^-(v', z) u'
\end{eqnarray}
for some constant $X_{1, n}^{n-1}$ to be determined later 
if $n \geq 1$, 
and 
\begin{eqnarray}
Y(v \otimes v', z) (u \otimes u') 
& = & 
\Phi^+(v, z) u \otimes \Psi^+(v', z) u'
\end{eqnarray}
if $n = 0$. 
(3.13) and (3.14) can certainly be unified by choosing $X_{1, 0}^{-1} =0$. 
Here, 
$\Phi^+(\cdot, z)$, $\Phi^-(\cdot, z)$ are intertwining operators of type 
$\left( \begin{array}{c} L(\triangle(n+1), c) \\ L(\triangle(1), c) \,\, L(\triangle(n), c) \end{array}  \right)$
and 
$\left( \begin{array}{c} L(\triangle(n-1), c) \\ L(\triangle(1), c) \,\, L(\triangle(n), c) \end{array}  \right)$
such that 
$$
v_{\triangle(1), c}^+(-1) v_{\triangle(n), c} = v_{\triangle(n+1), c}, \qquad
v_{\triangle(1), c}^-(-1) v_{\triangle(n), c} = v_{\triangle(n-1), c}. 
$$
Analogously, 
$\Psi^+(\cdot, z)$, $\Psi^-(\cdot, z)$ are intertwining operators of type 
$\left( \begin{array}{c} L(\bar \triangle(n+1), \bar c) \\ L(\bar \triangle(1), \bar c) \,\, L(\bar \triangle(n), \bar c) \end{array}  \right)$
and \\
$\left( \begin{array}{c} L(\bar \triangle(n-1), \bar c) \\ L(\bar \triangle(1), c) \,\, L(\bar \triangle(n), \bar c) \end{array}  \right)$
such that 
$$
v_{\bar \triangle(1), \bar c}^+(-1) v_{\bar \triangle(n), \bar c} = v_{\bar \triangle(n+1), \bar c}, \qquad
v_{\bar \triangle(1), \bar c}^-(-1) v_{\bar \triangle(n), \bar c} = v_{\bar \triangle(n-1), \bar c}. 
$$

In particular, 
we have
$$
Y({\bf v}_1, z) {\bf v}_1
= {\bf v}_2 + \ldots + X_{1, 1}^0 z^2 {\bf v}_0 + \ldots 
$$
Since no negative powers of $z$ appear in $Y({\bf v}_1, z) {\bf v}_1$, 
we expect the vertex operator $Y({\bf v}_1, z)$ to commute with itself, 
i.e. 
\begin{eqnarray}
[Y({\bf v}_1, z), Y({\bf v}_1, w)] & = & 0. 
\end{eqnarray}
We compute the following matrix coefficients using Proposition \ref{matrixcoefficients}:
\begin{eqnarray}
& & ({\bf v}_{n+2}^*, Y({\bf v}_1, z) Y({\bf v}_1, w) {\bf v}_n)
\\
& = & 
(v_{\triangle(n+2), c}^*,  \Phi^+(v_{\triangle(1), c}, z) \Phi^+ (v_{\triangle(1), c}, w) v_{\triangle(n), c})
\cdot (v_{\bar \triangle(n+2), \bar c}^*,  \Psi^+(v_{\bar \triangle(1), \bar c}, z) \Psi^+ (v_{\bar \triangle(1), \bar c}, w) v_{\bar \triangle(n), \bar c})
\nonumber \\
& = & 1. 
\nonumber
\end{eqnarray}

\begin{eqnarray}
& & ({\bf v}_{n}^*, Y({\bf v}_1, z) Y({\bf v}_1, w) {\bf v}_n)
\\
& = & 
X_{1, n+1}^n 
(v_{\triangle(n), c}^*,  \Phi^-(v_{\triangle(1), c}, z) \Phi^+ (v_{\triangle(1), c}, w) v_{\triangle(n), c})
\cdot (v_{\bar \triangle(n), \bar c}^*,  \Psi^-(v_{\bar \triangle(1), \bar c}, z) \Psi^+ (v_{\bar \triangle(1), \bar c}, w) v_{\bar \triangle(n), \bar c})
\nonumber \\
&  + &   
X_{1, n}^{n-1} 
(v_{\triangle(n), c}^*,  \Phi^+ (v_{\triangle(1), c}, z) \Phi^- (v_{\triangle(1), c}, w) v_{\triangle(n), c})
\cdot (v_{\bar \triangle(n), \bar c}^*,  \Psi^+ (v_{\bar \triangle(1), \bar c}, z) \Psi^- (v_{\bar \triangle(1), \bar c}, w) v_{\bar \triangle(n), \bar c})
\nonumber \\
& = &
X_{1, n+1}^n  z^2
{}_2F_1(\frac{1}{\varkappa}, \frac{n+2}{\varkappa} -1, \frac{n+1}{\varkappa}; \frac{w}{z})
{}_2F_1(-\frac{1}{\varkappa}, -\frac{n+2}{\varkappa} -1,  -\frac{n+1}{\varkappa}; \frac{w}{z})
\nonumber \\
& +  &
X_{1, n}^{n-1} w^2
{}_2F_1(1- \frac{n}{\varkappa}, \frac{1}{\varkappa},  2- \frac{n+1}{\varkappa}; \frac{w}{z})
{}_2F_1(1+ \frac{n}{\varkappa}, - \frac{1}{\varkappa},  2+ \frac{n+1}{\varkappa}; \frac{w}{z})
\nonumber 
\end{eqnarray}

\begin{eqnarray}
& & ({\bf v}_{n-2}^*, Y({\bf v}_1, z) Y({\bf v}_1, w) {\bf v}_n)
\\
& = & 
X_{1, n}^{n-1} X_{1, n-1}^{n-2} 
(v_{\triangle(n-2), c}^*,  \Phi^-(v_{\triangle(1), c}, z) \Phi^- (v_{\triangle(1), c}, w) v_{\triangle(n), c})
\nonumber \\
& & 
\cdot (v_{\bar \triangle(n-2), \bar c}^*,  \Psi^-(v_{\bar \triangle(1), \bar c}, z) \Psi^- (v_{\bar \triangle(1), \bar c}, w) v_{\bar \triangle(n), \bar c})
\nonumber \\
& = & 
X_{1, n}^{n-1} X_{1, n-1}^{n-2} z^2 w^2 
\nonumber 
\end{eqnarray}
By (3.15), 
we expect that 
\begin{eqnarray}
({\bf v}_{n}^*, Y({\bf v}_1, z) Y({\bf v}_1, w) {\bf v}_n) 
& = & 
({\bf v}_{n}^*, Y({\bf v}_1, w) Y({\bf v}_1, z) {\bf v}_n). 
\end{eqnarray}
Since the left-hand side does not involve negative powers of $w$ 
and the right-hand side does not involve negative powers of $z$, 
both of them must be (homogeneous) polynomials of degree $2$ in $z, w$. 
Therefore, 
the power series of $t = w/z$ in (3.17) must terminate, 
and that implies what the structure constants 
$X_{1, n+1}^n$, $X_{1, n}^{n-1}$ should be, 
or at least what their ratio is. 
We need the following hypergeometric identity: 

\begin{lemma}\label{hypergeom1}
\begin{eqnarray}
{}_2F_1(a, b, c, t) {}_2F_1(-a, -b-2, -c, t)
& & 
\\
+ \frac{a(b)_3(a-c)}{c(c-1)_3(c-b-1)} t^2 
{}_2F_1(a-c+1, b-c+1, 2-c, t) {}_2F_1(-a+c+1, -b+c-1, 2+c, t) 
& & 
\nonumber \\
= 1 -\frac{2a}{c} t + \frac{a(a-b-1)}{c(c-b-1)} t^2 
& &
\nonumber 
\end{eqnarray}
\end{lemma}

\begin{proof}
This is proved in Section 5. 
\end{proof}

Comparing (3.17) to (3.20), 
we choose 
\begin{eqnarray}
& & a = \frac{1}{\varkappa}, \quad 
b= \frac{n+2}{\varkappa} -1, \quad
c = \frac{n+1}{\varkappa}. 
\end{eqnarray}
Then $a + c = b + 1$, 
and
\begin{eqnarray}
& &
1 -\frac{2a}{c} t + \frac{a(a-b-1)}{c(c-b-1)} t^2 
= 1 - \frac{2}{n+1} t + t^2. 
\end{eqnarray}
(3.20) motivates us to require
\begin{eqnarray}
\frac{X_{1, n}^{n-1}}{X_{1, n+1}^n}
& = & 
\frac{\frac{1}{\varkappa} (\frac{n+2}{\varkappa} -1) \frac{n+2}{\varkappa} (\frac{n+2}{\varkappa} +1) (-\frac{n}{\varkappa})}
{\left( \frac{n+1}{\varkappa} \right)^2(\frac{n+1}{\varkappa} +1)(\frac{n+1}{\varkappa}-1)(-\frac{1}{\varkappa} )} 
\nonumber \\
& = & 
\frac{n}{(n+1) [(n+1)^2 -  \varkappa^2 ] } / 
\frac{n+1}{(n+2) [  (n+2)^2 -  \varkappa^2 ] }.  \nonumber 
\end{eqnarray}
We'll choose
\begin{eqnarray}
X_{1, n}^{n-1} = \frac{n}{n+1} \frac{1}{ (n+1)^2 - \varkappa^2 } 
\end{eqnarray}
for $n \geq 0$ 
in the definition of vertex operators (3.13) for elements of $W(1)$. 
The above calculations also showed

\begin{prop}
For $n \geq 2$, 
we have
\begin{eqnarray}
({\bf v}_{n+2}^*, Y({\bf v}_1, z) Y({\bf v}_1, w) {\bf v}_n)
& = & ({\bf v}_{n+2}^*, Y({\bf v}_1, w) Y({\bf v}_1, z) {\bf v}_n)
\quad = \quad 1, 
\\
({\bf v}_{n}^*, Y({\bf v}_1, z) Y({\bf v}_1, w) {\bf v}_n)
& = & 
({\bf v}_{n}^*, Y({\bf v}_1, w) Y({\bf v}_1, z) {\bf v}_n)
\quad = \quad 
X_{1, n+1}^n (z^2 - \frac{2}{n+1} z w + w^2 ), 
\\
({\bf v}_{n-2}^*, Y({\bf v}_1, z) Y({\bf v}_1, w) {\bf v}_n)
& = & ({\bf v}_{n-2}^*, Y({\bf v}_1, w) Y({\bf v}_1, z) {\bf v}_n)
\quad = \quad 
X_{1, n}^{n-1} X_{1, n-1}^{n-2} z^2 w^2. 
\end{eqnarray}
If $n=1$ or $0$, we only have the first two. 
\end{prop}

However, 
(3.24)-(3.26) are not sufficient to imply the locality of $Y({\bf v}_1, z)$ with itself. 
We need to consider more general matrix coefficients
\begin{eqnarray}
(L_{-j_1} \ldots \bar L_{-p_1} \ldots {\bf v}_m^*, Y({\bf v}_1, z) Y({\bf v}_1, w) L_{-i_1}\ldots \bar L_{-q_1} \ldots {\bf v}_n). 
\nonumber
\end{eqnarray}
Using the commutator relation
between $L_m, \bar L_m$ and $Y(L_{-1}^k \bar L_{-1}^{k'} {\bf v}_1, z)$, 
(3.27) can be expressed as linear combinations of matrix coefficients of the following type
\begin{eqnarray}
({\bf v}_m^*, Y(L_{-1}^k \bar L_{-1}^{k'} {\bf v}_1, z) Y(L_{-1}^l \bar L_{-1}^{l'} {\bf v}_1, w) {\bf v}_n). 
\end{eqnarray}
It turns out to be convenient to prove the locality of various fields of type 
$Y(L_{-1}^k \bar L_{-1}^{k'} {\bf v}_1, z)$ simultaneously. 

Let 
\begin{eqnarray}
A_n & = & 
\mathbb C[z^{\pm 1}, w^{\pm 1}](z-w)^{-n}, \qquad n \geq 1; 
\nonumber \\
A & = & A_0 \quad = \quad \mathbb C[z^{\pm 1}, w^{\pm 1}]. 
\nonumber 
\end{eqnarray}
For $v_m^* \in W(m)^*$, 
$v_n \in W(n)$, 
denote 
\begin{eqnarray}
\phi(v_m^*, v_n, k, k', l, l'; z, w) 
& = &
(v_m^*, Y(L_{-1}^{k} \bar L_{-1}^{k'} {\bf v}_1, z) 
Y(L_{-1}^{l} \bar L_{-1}^{l'} {\bf v}_1, w) v_n ). 
\end{eqnarray}
Our goal is to prove the following

\begin{prop} \label{extra}
For any $m, n, k, k', l, l' \geq 0$, 
$v_m^* \in W(m)^*$, 
and 
$v_n \in W(n)$, 
the matrix coefficients 
$$
\phi(v_m^*, v_n, k, k', l, l'; z, w) \qquad \text{and} \qquad 
\phi(v_m^*, v_n, l, l', k, k'; w, z)
$$
converge to the same rational function in $A_{k+k'+l+l'}$
in respective domains $|z| > |w|>0$ and $|w|>|z|>0$. 
Furthermore, 
for any $N \geq 0$, 
we have 
\begin{eqnarray}
& & 
\sum_{i=0}^N 
{N \choose i}
\phi(v_m^*, v_n, k+i, k', l+N-i, l'; z, w)
\in  A_{k+k'+l+l'}
\\
& & 
\sum_{i=0}^N 
{N \choose i}
\phi(v_m^*, v_n, k, k'+i, l, l'+N-i; z, w)
\in  A_{k+k'+l+l'}. 
\end{eqnarray}
\end{prop}

The above is stronger than the locality, 
in fact the locality of fields
$Y(L_{-1}^k \bar L_{-1}^{k'} {\bf v}_1, z)$
follows from the first half of the statement. 
We will prove Proposition \ref{extra} in 4 steps.

Denote 
$$
\Phi_1(z, w) = (v_{\triangle(n), c}^*,  \Phi^-(v_{\triangle(1), c}, z) \Phi^+ (v_{\triangle(1), c}, w) v_{\triangle(n), c})
$$
$$
\Phi_2(z, w) = (v_{\triangle(n), c}^*,  \Phi^+(v_{\triangle(1), c}, z) \Phi^- (v_{\triangle(1), c}, w) v_{\triangle(n), c})
$$
(see Proposition \ref{matrixcoefficients}), 
and the analogues of $\Phi_{1,2}$ in central charge $\bar c$ by $\Psi_{1, 2}$. 
Define
\begin{eqnarray}
\varphi_{k, l}(z, w) 
& = & (\partial_z^k \partial_w^l \Phi_1 ) \Psi_1 
+ \frac{ X_{1, n}^{n-1}}{X_{1, n+1}^n } (\partial_z^k \partial_w^l \Phi_2 ) \Psi_2. 
\end{eqnarray}
Then 
\begin{eqnarray}
({\bf v}_n^*, Y(L_{-1}^k {\bf v}_1, z) Y(L_{-1}^l {\bf v}_1, w) {\bf v}_n)
= X_{1, n+1}^n \varphi_{k, l}(z, w). 
\end{eqnarray}
{\bf Step 1} is to show $\varphi_{k, l}(z, w) = \varphi_{l, k}(w, z) \in A_{k+l}$. 
We already know that 
\begin{eqnarray}
\varphi_{0, 0}(z, w) = z^2 - \frac{2}{n+1} z w + w^2. 
\end{eqnarray}
To compute $\varphi_{1, 0}(z, w)$, 
we need another identity of hypergeometric series: 

\begin{lemma}\label{hypergeom2}
\begin{eqnarray}
 {}_2F_1(a, b, c, t) {}_2F_1(-a, -b-2, -c, t)  & &
\\
+ \frac{ab}{c(c-1)} t {}_2F_1 (a+1, b+1, c+1, t) {}_2F_1 (-a, -b-2, -c, t) & &
\nonumber \\
+ \frac{a(b)_3(a-c)_2 } {(c-2)_3 (c-1)_3}  t^3 
 {}_2F_1(a-c+2, b-c+2, 3-c, t) {}_2F_1(-a+c+1, -b+c-1, 2+c, t) & &
\nonumber \\
 =1 + \frac{a(b-2c+2)}{c(c-1)} t + \frac{a(a-b-1)}{c(c-1)} t^2. & &
\nonumber
\end{eqnarray}
\end{lemma}

\begin{proof}
This is proved in Section 5. 
\end{proof}

Let $a, b, c$ take the values specified in (3.21) and let 
\begin{eqnarray}
& & (\alpha_1, \beta_1) = ( 1-\frac{n+3}{2\varkappa},  \frac{n}{2\varkappa}), 
\qquad 
(\alpha_2, \beta_2)  = ( \frac{n-1}{2\varkappa}, 1- \frac{n+2}{2\varkappa}), 
\qquad
\gamma = \frac{1}{2 \varkappa}. 
\end{eqnarray}
Using (3.20) and (3.34), 
it is straightforward to check that 
\begin{eqnarray}
\varphi_{1, 0}(z, w) 
& = &
(1-c) z \left[1 + \frac{a(b-2c+2)}{c(c-1)} t + \frac{a(a-b-1)}{c(c-1)} t^2 \right] 
\\
& + &  
z \left( \alpha_2 + \frac{\gamma t}{1-t} \right) \left[ 1 -\frac{2a}{c} t + \frac{a(a-b-1)}{c(c-b-1)} t^2 \right] 
\nonumber 
\end{eqnarray}
and 
\begin{eqnarray}
\varphi_{0, 1}(z, w) 
& = &
\frac{(c-1)z^2}{w} \left[1 + \frac{a(b-2c+2)}{c(c-1)} t + \frac{a(a-b-1)}{c(c-1)} t^2 \right] 
\\
& + &  
\frac{z^2}{w} \left( \beta_2 - \frac{\gamma t}{1-t} \right) \left[ 1 -\frac{2a}{c} t + \frac{a(a-b-1)}{c(c-b-1)} t^2 \right]. 
\nonumber 
\end{eqnarray}
Both $\varphi_{1, 0}(z, w) $ and $\varphi_{0, 1}(z, w)$ are viewed 
as power series expansions in the region $|z| > |w| >0$. 
It is also convenient to regard them as rational functions of $z$ and $w$ 
with poles at $z = 0$, $w =0$, and $z=w$. 

\begin{lemma}\label{phi1001}
We have $\varphi_{1, 0}, \varphi_{0, 1} \in A_1$, 
$\varphi_{1, 0} + \varphi_{0, 1} \in A$, 
and $\varphi_{1, 0}(z, w) = \varphi_{0, 1}(w, z)$. 
\end{lemma}

\begin{proof}
This is shown by direct calculation. 
The details are omitted.
\end{proof}



To compute $\varphi_{1, 1}(z, w)$, 
we need the following identity of hypergeometric series: 

\begin{lemma} \label{hypergeom3}
\begin{eqnarray}
{}_2F_1(a, b, c, t) {}_2F_1 (-a, -b-2, -c, t) & &
\\
+\frac{2ab}{c(c-2)} t \, {}_2F_1(a+1, b+1, c+1, t) {}_2F_1(-a, -b-2, -c, t) & &
\nonumber \\
+\frac{(a)_2(b)_2}{(c-2)_4} \, t^2  {}_2F_1(a+2, b+2, c+2, t) {}_2F_1(-a, -b-2, -c, t) & &
\nonumber \\
+\frac{a(b)_3(a-c)_3(c-b-2)}{(c-3)_4(c-2)_4}  
t^4\, {}_2F_1 (a-c+3, b-c+3, 4-c, t) {}_2F_1(-a+c+1, -b+c-1, c+2, t) & &
\nonumber \\
=1 + \frac{2a(b-c+2)}{c(c-2)} t + \frac{a(a-b-1)(c-b-2)}{(c-2)(c-1)c} t^2 & &
\nonumber
\end{eqnarray}
\end{lemma}

\begin{proof}
This is proved in Section 5. 
\end{proof}

Again, 
it is straightforward to check that 
\begin{eqnarray}
& & \varphi_{1, 1}(z, w) 
\\
& = & 
\left[
\frac{\alpha_2 \beta_2}{t} + \frac{(\beta_2 - \alpha_2 +1) \gamma}{1-t} 
- \frac{\gamma(\gamma-1)t}{(1-t)^2}
\right]
\left[
1 -\frac{2a}{c} t + \frac{a(a-b-1)}{c(c-b-1)} t^2
\right]
\nonumber \\
& + & 
(1-c) \left[
\frac{\beta_2 - \alpha_2 +1}{t} - \frac{2 \gamma}{1-t} 
\right]
\left[
1 + \frac{a(b-2c+2)}{c(c-1)} t + \frac{a(a-b-1)}{c(c-1)} t^2
\right]
\nonumber \\
& - & 
\frac{(1-c)(2-c)}{t}
\left[
1 + \frac{2a(b-c+2)}{c(c-2)} t + \frac{a(a-b-1)(c-b-2)}{(c-2)(c-1)c} t^2
\right]
\nonumber 
\end{eqnarray}
using (3.20), (3.34), and (3.38). 

\begin{lemma} \label{phi11}
We have $\varphi_{1, 1}(z, w)  \in A_2$ and $\varphi_{1, 1}(z, w) = \varphi_{1, 1}(w, z)$. 
\end{lemma}

\begin{proof}
This is shown by direct calculation. The details are omitted.
\end{proof}

\begin{corollary} \label{phikl}
For any $k, l \geq 0$, 
we have 
$\varphi_{k, l}(z, w) \in A_{k+l}$. 
Furthermore 
$\varphi_{k, l}(z, w) = \varphi_{l, k} (w, z)$. 
\end{corollary}

\begin{proof}
Since $\Phi_{1, 2}$ satisfy the equations (3.1)-(3.2), 
any $\varphi_{k, l}(z, w)$ can be expressed in terms of 
$\varphi_{0, 0}$, $\varphi_{1, 0}$, $\varphi_{0, 1}$, and $\varphi_{1, 1}$, 
and has a pole of order at most $k+l$ at $z=w$ . 
Lemma \ref{phi1001}, \ref{phi11}, 
and the symmetry of (3.1) and (3.2) with respect to the interchange of $z$ and $w$ 
imply that $\varphi_{k, l}(z, w) = \varphi_{l, k} (w, z)$.
\end{proof}

\noindent
{\bf Step 2} is to show that 
$\sum_{i=0}^N {N \choose i} \varphi_{k+i, l+N-i} \in A_{k+l}$ for any $k, l, N\geq 0$. 

\begin{lemma} \label{bino}
Let $t = w/z$. 
Then 
\begin{eqnarray} 
& &  
\sum_{i=0}^N {N \choose i} \partial_z^i \partial_w^{N-i} [z^\Lambda h(t)]
= 
z^{\Lambda -N} \sum_{j=0}^N {N \choose j} (\Lambda -N +1)_j (1-t)^{N-j} \partial_t^{N-j} h(t).
\end{eqnarray}
\end{lemma}

\begin{proof}
First, 
we show that
\begin{eqnarray}
\partial_z^n ( z^\Lambda h(t)) 
& = & 
\sum_{i=0}^n (-1)^i { n \choose i} (\Lambda - n+1)_{n-i} z^{\Lambda -n-i} w^i \partial^i_t h(t). 
\end{eqnarray}
The rest is straightforward. 
Details are omitted. 
\end{proof}

\begin{corollary} \label{blah}
For any $k, l, N \geq 0$, 
we have
$$
\sum_{i=0}^N {N \choose i} \varphi_{k+i, l+N-i} \in A_{k+l}. 
$$
\end{corollary}

\begin{proof}
Write $\Phi_{1, 2}(z, w) = z^{\alpha_1 + \beta_1} h_{1, 2}(t)$. 
Then 
\begin{eqnarray}
h_1^{(i)} (t) \Psi_1 + \frac{X_{1, n}^{n-1}}{X_{1, n+1}^n}  h_2^{(i)} (t) \Psi_2  
& = & 
z^{-\alpha_1-\beta_1+i} 
\left(
(\partial_w^i \Phi_1) \Psi_1 + \frac{X_{1, n}^{n-1}}{X_{1, n+1}^n} (\partial_w^i \Phi_2) \Psi_2
\right)
 \\
& = & 
z^{-\alpha_1-\beta_1+i} \varphi_{0, i} 
\in A_i. 
\nonumber 
\end{eqnarray}
Then, 
\begin{eqnarray}
\sum_{i=0}^N {N \choose i} \varphi_{k+i, l+N-i} 
& = & 
\partial_z^k \partial_w^l 
\left[ \sum_{i=0}^N {N \choose i} \partial_z^i \partial_w^{N-i} (z^{\alpha_1+\beta_1} h_1(t)) \right] \Psi_1
\\
& + & 
\frac{X_{1, n}^{n-1}}{X_{1, n+1}^n}
\partial_z^k \partial_w^l 
\left[ \sum_{i=0}^N {N \choose i} 
\partial_z^i \partial_w^{N-i} 
(z^{\alpha_1+\beta_1} h_2(t)) \right]  \Psi_2. 
\nonumber
\end{eqnarray}
Using Lemma \ref{bino} and (3.42), 
it is not difficult to see that (3.43) belongs to $A_{k+l}$. 
\end{proof}

\noindent
{\bf Step 3} is to show that Proposition \ref{extra} holds for matrix coefficients
determined by all pairs of (${\bf v}_m^*$, ${\bf v}_n$), $m, n \geq 0$. 
Step 1 and Step 2 already showed that the first half of Proposition \ref{extra} and (3.29) 
are true for matrix coefficients determined by 
(${\bf v}_n^*$, ${\bf v}_n$)
and $k' =l' =0$. 
It is easier to verify that the same are true for other pairs of 
(${\bf v}_m^*$, ${\bf v}_n$) and $k'=l'=0$, 
where $m= n+2$ or $n-2$ if $n\geq 2$, and only $n+2$ if $n \leq 1$, 
because they don't involve any hypergeometric series. 
Let 
$$
D = L_{-1} + \bar L_{-1}
$$ 
be the translation operator. 
Then $[D, L_{-1}] = [D, \bar L_{-1}] = 0$ and for any $v \in W(1)$, 
we have 
$$
[D, Y(v, z) ] = Y( D v, z) = \partial_z Y(v, z). 
$$
Therefore,  
the matrix coefficient 
\begin{eqnarray}
({\bf v}_m^*, Y(L_{-1}^k \bar L_{-1}^{k'} {\bf v}_1, z) Y(L_{-1}^l \bar L_{-1}^{l'} {\bf v}_1, w) {\bf v}_n ) 
\end{eqnarray}
can be expressed as linear combinations of 
$$
\partial_z^{k'-p} \partial_w^{l' -q} 
({\bf v}_m^*, Y(L_{-1}^{k+p} {\bf v}_1, z) Y(L_{-1}^{l+q} {\bf v}_1, w) ] {\bf v}_n ),
$$
where $0 \leq p \leq k', 0 \leq q \leq l'$. 
Then, 
one verifies that 
Proposition \ref{extra} is in fact true for all matrix coefficients determined by the highest weight vectors
${\bf v}_m^*$, ${\bf v}_n$. 
The details are tedious, therefore omitted. 

\medskip
\noindent
{\bf Step 4} is the induction part. 
Suppose Proposition \ref{extra} is true for the pair $(v_m^*, {\bf v}_n)$, 
we will show that it is true for $(L_{-p} v_m^*, {\bf v}_n)$, where $p>0$. 
The first half is easy to show once we establish the second half, i.e.
\begin{eqnarray}
\sum_{i=0}^N {N \choose i} 
\phi(L_{-p} v_m^*, {\bf v}_n, k+i, k', l+N-i, l'; z, w) & \in & A_{k+l}.
\end{eqnarray}
We need the commutator relation
\begin{eqnarray}
[L_p,  Y(L_{-1}^k {\bf v}_1, z) ] 
& = & 
\sum_{q \geq 0} z^{p+1-q} {p+1 \choose q} Y(L_{q-1} L_{-1}^k {\bf v}_1, z), 
\end{eqnarray}
which follows from the property of intertwining operators and the definition (3.13)-(314) of $Y$. 
We also need the following two lemmas:

\begin{lemma}\label{l1}
For any $k  \geq 0, q \geq 0$,
we have  
$$
L_{q-1} (L_{-1}^k {\bf v}_1)
= q! \left[ {k \choose q-1} \triangle(1)  + {k \choose q} \right] 
L_{-1}^{k-q+1} {\bf v}_1. 
$$
\end{lemma}

\begin{proof}
The proof is straightforward.
\end{proof}

\begin{lemma}\label{l2}
For any $0 \leq i \leq N$, $q \geq 0$, 
we have 
$$
{N \choose i} {k +i \choose q} = \sum_{j=0}^q {N \choose j} {k \choose q-j} {N-j \choose N-i}. 
$$
\end{lemma}

\begin{proof}
This can be proved by taking the coefficient of $y^{N-i} x^q$ in the power series expansion of the polynomial 
$(1+x)^k (1+x+y)^N$ in two different ways:
$(1+x)^k ( (1+x) +y)^N$
and 
$(1+x)^k (x+ (1+ y))^N$. 
\end{proof}

Using (3.46) and Lemma \ref{l1}, 
we can write the left-hand-side of (3.45) as 
$$ I + II + III + IV \quad (\text{mod  }  A_{k+l})
$$
where 
\begin{eqnarray}
I & = & 
\sum_{q =0}^N 
{p+1 \choose q}  z^{p+1-q} q! \triangle(1)
\sum_{i= \text{max} (0, q-k-1) }^N 
{N \choose i}
{k+i \choose q-1} 
\phi(v_m^*, {\bf v}_n^*, k+i-q+1, k', l+N-i, l'; z, w)
\nonumber \\
II & = &
\sum_{q=0}^N 
{p+1 \choose q}  z^{p+1-q} q! 
\sum_{i= \text{max} (0, q-k-1) }^N 
{N \choose i}
{k+i \choose q} 
\phi(v_m^*, {\bf v}_n^*, k+i-q+1, k', l+N-i, l'; z, w)
\nonumber 
\end{eqnarray}
\begin{eqnarray}
III & =& 
\sum_{q = 0}^N {p+1 \choose q} w^{p+1-q} q!
\triangle(1) 
\sum_{i=0}^{ \text{min} (N, l + N -q +1)}  
{N \choose i}
{l+N-i  \choose q-1} 
\phi(v_m^*, {\bf v}_n^*, k+i, k', l+ N -i -q +1, l'; z, w)
\nonumber \\
IV & = &
\sum_{q = 0}^N {p+1 \choose q} w^{p+1-q} q!
\sum_{i=0}^{ \text{min} (N, l + N -q +1)}  
{N \choose i}
{l+N-i  \choose q} 
\phi(v_m^*, {\bf v}_n^*, k+i, k', l+ N -i -q +1, l'; z, w). 
\nonumber
\end{eqnarray}
Then using Lemma \ref{l2} and the induction hypothesis, 
one shows that 
$I, II, III + IV \in A_{k+l}$.
The details are omitted. 
The same arguments apply when $L_{-p}$ is replaced by $\bar L_{-p}$. 
Once we prove Proposition \ref{extra} for all matrix coefficients of type 
$(v_m^*, {\bf v}_n)$, 
we implement the same induction process to the second entry, 
and hence prove the proposition for all $(v_m^*, v_n)$.

\section{Induction construction of vertex operators}

In Section 3, 
we defined the vertex operators associated to elements of $W(1)$ and proved the locality.  
Now, 
we use the Reconstruction Theorem to extend the structure to the whole space
$$
W = \bigoplus_{n \geq 0} W(n), \qquad W(n) = L(\triangle(n), c) \otimes L(\bar \triangle(n), \bar c). 
$$

\begin{theorem}{[FB]}  \label{reconstruction}
(Strong Reconstruction Theorem)
Let $V$ be a vector space, ${\bf 1}$ a non-zero vector, and $D$ an endomorphism of $V$. 
Let $\{a^s\}_{s \in S}$ be a collection of vectors in $V$. 
Suppose we are also given fields
$$
a^s(z) = \sum_{n \in \mathbb Z} a^s(n) z^{-n-1}
$$
such that 
\begin{enumerate}
\item 
For all $x$, $a^s(z) {\bf 1} = a^s + z( \ldots)$. 

\item 
$D{\bf 1} =0$ and $[D, a^s(z)] = \partial_z a^s(z)$. 

\item 
All fields $a^s(z)$ are mutually local. 

\item 
$V$ is spanned by the vectors
$$
a^{s_1}(j_n) \ldots a^{s_n}(j_n) {\bf 1}, \quad  j_i <0. 
$$
\end{enumerate}
Then these structures together with 
the assignment
$$
Y(a^{s_1}(j_1)\ldots a^{s_n}(j_n) {\bf 1}, z)
= \frac{1}{(-j_1-1)! \ldots (-j_n-1)! }
: \partial_z^{-j_1-1} a^{s_1}(z) \ldots \partial_z^{-j_n-1} a^{s_n}(z) :
$$
give rise to a well-defined vertex algebra structure on $V$. 
\end{theorem}


Here, 
to apply the theorem to $W$, 
we take ${\bf 1} = {\bf v}_0 \in W(0)$, 
$D = L_{-1} + \bar L_{-1}$, 
and 
$$
\{a^s\}_{s \in S} = \{ L_{-2} {\bf v}_0, \bar L_{-2} {\bf v}_0, L_{-1}^k \bar L_{-1}^{k'} {\bf v}_1\}_{k, k' \in \mathbb N}. 
$$ 
The fields associated to $L_{-2} {\bf v}_0$, $\bar L_{-2} {\bf v}_0$
are 
$$
L(z)=\sum_{n \in \mathbb Z} L_n z^{-n-2} 
\qquad \text{and} \qquad 
\bar L(z) = \sum_{n \in \mathbb Z} \bar L_n z^{-n-2}. 
$$
Fields associated to $L_{-1}^k \bar L_{-1}^{k'} {\bf v}_1$ and in general, 
elements of $W(1)$, 
are given in (3.13)-(3.14) with the structure constants $X_{1, n}^{n-1}$ we've chosen (3.23). 
The locality among them are established in Section 3. 
To see (4) holds, note that 
$$
{\bf v}_2 = {\bf v}_1(-1) {\bf v}_1, 
$$
and in general 
$$
L_{-1}^k {\bf v}_2 = \sum_{i = 0}^k {k \choose i} (L_{-1}^i{\bf v}_1)(-1) (L_{-1}^{k-i} {\bf v}_1), 
\quad k \geq 0. 
$$
Furthermore, 
we have 
$$
{\bf v}_3 = {\bf v}_1 (-1) {\bf v}_2, \qquad
L_{-1}^k {\bf v}_3 = \sum_{i = 0}^k {k \choose i} (L_{-1}^i{\bf v}_1)(-1) (L_{-1}^{k-i} {\bf v}_2), 
\quad \text{etc.} 
$$ 
So $W$ is spanned by vectors obtained by applying the operators 
$$
L_{-i}, \hspace{2mm} \bar L_{-i}, \quad i \geq 2,  
\quad \quad 
(L_{-1}^k \bar L_{-1}^{k'} {\bf v}_1) (-1)
$$
to the vacuum ${\bf v}_0$. 
By Theorem \ref{reconstruction}, 
we have the following 

\begin{corollary}
The space
$$
W = \bigoplus_{n \geq 0} L(\triangle(n), c) \otimes L(\bar \triangle(n), \bar c)
$$
is endowed with the structure of a vertex algebra. 
It is in fact a vertex operator algebra of rank $26$ where the Virasoro element is given by 
$L_{-2} {\bf v}_0 + \bar L_{-2} {\bf v}_0$. 
\end{corollary}

Note that the $\mathbb Z$-grading on $W$ is determined by 
$\text{deg  }  {\bf v}_n = -n$
and the fact that $L_{-i}, \bar L_{-i}$ have degree $i$. 
In particular, the grading is not bounded from below. 

Recall that the vertex operators associated to elements of $W(1)$ 
are defined by pairing the intertwining operators together using the constants
$$
X_{1, n}^{n+1} = 1, \qquad 
X_{1, n}^{n-1} = \frac{n}{n+1} \frac{1}{(n+1)^2 -\varkappa^2}, \quad n \geq 0. 
$$
In this section, 
we will derive analogous formulas 
for the vertex operators 
associated to any element of $W$.
We do so by induction. 

Suppose the vertex operators for elements of $W(\lambda)$, $\lambda \geq 1$, 
are given by the linear combination of tensor products of intertwining operators
with the appropriate structure constants 
$X_{\lambda, \mu}^\nu$, 
$\mu, \nu \geq 0$, 
we will show that the same is true for elements of $W(\lambda+1)$. 
Meanwhile, 
we will derive a recursion for the set of constants 
$X_{\lambda, \mu}^\nu$, 
and the explicit formula. 

Define 
\begin{eqnarray}
\Phi^{-+} (z, w) 
& = & 
 (v_{\triangle(\nu), c}^*, \Phi_{1, \nu+1}^\nu (v_{\triangle(1), c}, z) 
\Phi_{\lambda, \mu}^{\nu+1} (v_{\triangle(\lambda), c}, w) v_{\triangle(\mu), c}), 
\\
\Phi^{+-} (z, w) 
& = & 
(v_{\triangle(\nu), c}^*, \Phi_{1, \nu-1}^\nu (v_{\triangle(1), c}, z) 
\Phi_{\lambda, \mu}^{\nu-1} (v_{\triangle(\lambda), c}, w) v_{\triangle(\mu), c})
\end{eqnarray}
here 
$\Phi_{\lambda, \mu}^\nu(\cdot, z)$ is the intertwining operator of type 
${L(\triangle(\nu), c) \choose L(\triangle(\lambda), c) \, L(\triangle(\mu), c) }$, 
whenever it exists, 
such that 
$$
\Phi_{\lambda, \mu}^\nu(v, z) = z^{\triangle(\nu) -\triangle(\lambda)-\triangle(\mu)} 
\sum_{n \in \mathbb Z} v(n) z^{-n-1}
$$
with 
$v_{\triangle(\lambda), c}(-1) v_{\triangle(\mu), c} = v_{\triangle(\nu), c}$. 

\begin{lemma}
The $\Phi^{\mp \pm} (z, w) $ satisfy the following equation
\begin{eqnarray}
\partial_z^2 \Phi = \frac{1}{\varkappa} \left(
\frac{w}{z(z-w)} \partial_w - \frac{1}{z} \partial_z + \frac{\triangle(\lambda)}{(z-w)^2} + \frac{\triangle(\mu)}{z^2} \right) \Phi
\end{eqnarray}
\end{lemma}

\begin{proof}
The proof is the same as (3.1), which is a special case with $\lambda =1$. 
\end{proof}

The equation (4.20) comes from the singular vector 
of weight $(\triangle(1)+2, c)$ in $M(\triangle(1), c)$. 
Even though an explicit formula for the nontrivial singular vector of $M(\triangle(\lambda), c)$ 
is known
(see Lemma \ref{sing}), 
we will not derive the corresponding differential equation here, 
which would generalize (3.2) and have $\partial_w^{\lambda +1} \Phi$ on one side. 

\begin{lemma}
We have 
\begin{eqnarray}
\Phi^{-+} (z, w) & = & 
z^{\alpha_1} w^{\beta_1} (1-t)^\gamma 
{}_2F_1( a, b, c; t), 
\\
\Phi^{+-} (z, w) & = & 
z^{\alpha_2} w^{\beta_2} (1-t)^\gamma 
{}_2 F_1 (a-c+1, b-c+1, 2-c; t)
\\
& = & 
z^{\alpha_1} w^{\beta_1} (1-t)^\gamma \, \, t^{1-c}
{}_2 F_1 (a-c+1, b-c+1, 2-c; t), 
\nonumber
\end{eqnarray}
where the constants $\alpha_{1, 2}$, $\beta_{1, 2}$, 
$\gamma$, $a$, $b$, $c$ are given as follows:
\begin{eqnarray}
\alpha_1 & = &  \triangle(\nu) - \triangle(\nu+1) - \triangle(1) = 1 - \frac{\nu+3}{2\varkappa} \\
\beta_1 & = &  \triangle(\nu+1) - \triangle(\lambda) - \triangle(\mu) \\
\alpha_2 & = & 
\triangle(\nu) - \triangle(\nu-1) - \triangle(1) = \frac{\nu-1}{2 \varkappa}
\\
\beta_2 & = &  \triangle(\nu-1) - \triangle(\lambda) - \triangle(\mu). 
\\
\gamma & = & \frac{\lambda}{2\varkappa} \\
a & = & \frac{\lambda - \mu + \nu + 1}{2 \varkappa}, \\
b & = & \frac{ \lambda + \mu +  \nu +3 }{2 \varkappa} -1, \\
c & = & \frac{ \nu+1}{\varkappa}. 
\end{eqnarray}
\end{lemma}

\begin{proof}
Note that 
\begin{eqnarray}
\Phi^{-+} (z, w) & = & z^{\alpha_1} w^{\beta_1} \phi^{-+}(t) 
\nonumber \\
\Phi^{+-} (z, w) & =  & z^{\alpha_2} w^{\beta_2} \phi^{+-} (t)
\nonumber
\end{eqnarray}
where $t = w/z$, $\phi^{\mp \pm} (t) \in \mathbb C[[t]]$ with $\phi^{\mp \pm} (0) =1$. 
When we plug 
$$\Phi = z^\alpha w^\beta (1-t)^\gamma \phi(t)$$
into the equation (4.3), 
where 
$(\alpha, \beta) = (\alpha_1, \beta_1)$ 
or $(\alpha_2, \beta_2)$, 
we obtain hypergeometric equations. 
The solutions then follow. 
\end{proof}

Furthermore, we have
\begin{eqnarray}
\Psi^{-+} & = & 
z^{2- \alpha_1} w^{\lambda + \mu- \nu -1 - \beta_1} (1-t)^{- \gamma} 
{}_2F_1( -a, -2 - b, - c; t), 
\\
\Psi^{+-} & = & 
z^{ - \alpha_2} w^{ \lambda + \mu - \nu +1 - \beta_2} (1-t)^{- \gamma} 
{}_2 F_1 ( - a + c + 1, - b + c -1, 2 + c; t), 
\end{eqnarray}
where $\Psi^{\mp \pm}$ are the corresponding correlation functions 
of the intertwining operators for the other copy of the Virasoro action.

Suppose all the intertwining operators in (4.18)-(4.19) exist for some $\lambda, \mu, \nu$,  
then we have
\begin{eqnarray}
& & ({\bf v}_\nu^*, Y({\bf v}_1, z) Y({\bf v}_\lambda, w) {\bf v}_\mu) 
\nonumber \\
& = & 
X_{1, \nu+1}^\nu X_{\lambda, \mu}^{\nu+1} \Phi^{-+} \Psi^{-+}
+ X_{\lambda, \mu}^{\nu-1} \Phi^{+-} \Psi^{+-}
\nonumber \\
& = & 
X_{1, \nu+1}^\nu X_{\lambda, \mu}^{\nu+1} 
z^2 w^{\lambda + \mu - \nu -1} 
{}_2F_1(a, b, c; t) {}_2 F_1(-a, -b-2, -c; t) 
\nonumber \\
& +  & 
X_{\lambda, \mu}^{\nu-1} 
w^{\lambda + \mu - \nu +1}
{}_2F_1(a-c+1, b-c+1, 2-c; t) 
{}_2F_1(-a+c+1, -b+c-1, 2+c; t). 
\nonumber 
\end{eqnarray}

Since $Y( {\bf v}_1, z)$ and $Y({\bf v}_{\lambda}, w)$ commute with each other, 
the above power series of $t =w/z$ terminates.
If we were to guess, 
we would expect that 
\begin{eqnarray}
 \frac{ X_{\lambda, \mu}^{\nu-1} }
{ X_{1, \nu+1}^\nu X_{\lambda, \mu}^{\nu+1} } 
& = & 
\frac{a b(b+1)(b+2)(a-c)} {c^2 (c+1)(c-1)(c-b-1)}
\\
& = & 
- \frac{ 
\frac{ \lambda - \mu + \nu +1 }{2} 
\frac{\lambda - \mu - \nu -1}{2}
\frac{\lambda + \mu + \nu +3}{2}}
{ (\nu+1)^2 \frac{\lambda + \mu - \nu +1}{2} }
\frac{ 
\left( \frac{\lambda + \mu + \nu + 3}{2} \right) ^2 - \varkappa^2 }
{ (\nu+1)^2 - \varkappa^2 }, 
\nonumber
\end{eqnarray}
and 
\begin{eqnarray}
& &  ({\bf v}_\nu^*, Y({\bf v}_1, z) Y({\bf v}_\lambda, w) {\bf v}_\mu) \\
 & = & 
X_{1, \nu+1}^\nu X_{\lambda, \mu}^{\nu+1} 
z^2 w^{\lambda + \mu - \nu -1} 
\left( 1 - \frac{2a}{c} t + \frac{a(a-b-1)}{c(c-b-1)} t^2 \right)
\nonumber \\
& = & 
X_{1, \nu+1}^\nu X_{\lambda, \mu}^{\nu+1} 
z^2 w^{\lambda + \mu - \nu -1} 
\left(
1 - \frac{\lambda -\mu + \nu +1}{\nu+1} t
+ \frac{(\mu+1)(\lambda -\mu + \nu +1)}{(\nu+1)(\lambda+ \mu - \nu +1)} 
t^2
\right) 
\nonumber
\end{eqnarray}
(see Lemma \ref{hypergeom1}). 
Since ${\bf v}_{\lambda+1} = {\bf v}_1 (-1) {\bf v}_{\lambda}$, 
and $Y( {\bf v}_1, z)$, 
$Y({\bf v}_{\lambda}, w)$ 
commute with each other, 
we have
\begin{eqnarray}
Y({\bf v}_{\lambda+1}, w) & = & [Y( {\bf v}_1, z) Y({\bf v}_{\lambda}, w)]|_{z=w}, 
\nonumber
\end{eqnarray}
hence,
by (4.17), 
\begin{eqnarray}
({\bf v}_\nu^*,  Y({\bf v}_{\lambda+1}, w) {\bf v}_\mu) 
& = &
- \frac{  ( \lambda+1) (\lambda - \mu - \nu -1) }{ (\nu+1) (\lambda +\mu - \nu +1) } 
X_{1, \nu+1}^\nu X_{\lambda, \mu}^{\nu+1} 
w^{\lambda + \mu - \nu + 1}.  
\nonumber 
\end{eqnarray}
If the vertex operators for elements of $W(\lambda+1)$ 
were indeed made up of intertwining operators with structure constants 
$X_{\lambda+1, \mu}^\nu$, 
we would again guess the following recursion: 
\begin{eqnarray}
X_{\lambda +1, \mu}^\nu 
& = & 
- \frac{  ( \lambda+1) (\lambda - \mu - \nu -1) }{ (\nu+1) (\lambda +\mu - \nu +1) } 
X_{1, \nu+1}^\nu X_{\lambda, \mu}^{\nu+1}. 
\end{eqnarray}
The values (3.23) of $X_{1, \nu+1}^\nu$ for $\nu \geq 0$
and the recursion determine the following set of constants:

\begin{define} \label{structureconstants}
For $\lambda, \mu, \nu \geq 0$ such that 
$$
\lambda + \mu + \nu  \equiv   0 \,\, (\text{  mod  } 2) 
$$
and 
$$
0 \leq \ell = \frac{ \lambda + \mu - \nu} {2} \leq \lambda,
\qquad 
\text{i.e.     } \quad \nu \leq \lambda + \mu \quad  \text{and} \quad \mu \leq \lambda + \nu, 
$$
define
\begin{eqnarray}
X_{\lambda, \mu}^{\nu} 
& = &
\frac{ {\lambda \choose \ell} { \mu \choose \ell} }{ {  \frac{\lambda + \mu + \nu +2}{2} \choose \ell} }
\frac{1}{ [ (\nu+2)^2 - \varkappa^2 ]  [ (\nu+3)^2 - \varkappa^2 ] \ldots 
 [ ( \frac{\lambda + \mu + \nu +2}{2} )^2 - \varkappa^2 ] }; 
\end{eqnarray}
otherwise $X_{\lambda, \mu}^\nu = 0$. 
\end{define}

Note that if $\mu < \ell$, 
i.e. $\nu < \lambda - \mu$ (which particularly implies that $\lambda > \mu$), 
then $X_{\lambda, \mu}^\nu = 0$.
Therefore, 
$X_{\lambda, \mu}^\nu = 0$ unless 
$\lambda + \mu + \nu  \equiv   0 \,\, (\text{  mod  } 2)$
and 
$\lambda \leq \mu + \nu$, 
$\mu \leq \lambda + \nu$, 
$\nu \leq \lambda + \mu$. 
It is straightforward to verify that
(4.19) satisfies (4.16). 

The induction hypothesis can now be restated as follows:
suppose the vertex operators 
$Y(v \otimes v', z)$ for $v \otimes v' \in W(\lambda)$
are given by 
$$
Y(v \otimes v', z) u \otimes u' = 
\sum_{\nu \geq 0} X_{\lambda, \mu}^\nu
\Phi_{\lambda, \mu}^\nu(v, z) u 
\otimes
\Psi_{\lambda, \mu}^\nu (v', z) u'
$$
where $u \otimes u' \in W(\mu)$, 
$X_{\lambda, \mu}^\nu$ are the constants defined in (4.19), 
and 
$\Phi_{\lambda, \mu}^\nu(\cdot, z)$, 
$\Psi_{\lambda, \mu}^\nu (\cdot, z)$
are normalized intertwining operators for the two Virasoro actions. 
To prove it also holds for $W(\lambda+1)$, 
our strategy is to show that 
the vertex operators for elements of  $W(\lambda+1)$, 
which result from the Reconstruction Theorem, 
and the proposed formula
using intertwining operators and the structure constants
$X_{\lambda+1, \mu}^\nu$, $\mu, \nu \geq 0$, 
have the same matrix coefficients on the level of 
highest weight vectors, and 
the same commutator relations with the Virasoro actions, 
therefore they must be the same. 
We will do this in steps.

Denote
\begin{eqnarray}
\phi^\lambda(v_m^*, v_n, k, k', l, l'; z, w) 
& = &
(v_m^*, Y(L_{-1}^k \bar L_{-1}^{k'} {\bf v}_1, z)
Y(L_{-1}^l \bar L_{-1}^{l'} {\bf v}_\lambda, w)  v_n)
\end{eqnarray}
(see (3.28)). 

\medskip 
\noindent
{\bf Step 1:} 
The goal of step 1 
is to compute
$[(z-w)^{k+k'+l+l'} \sum_{i=0}^N
\sum_{j=0}^M
 {N \choose i} 
 {M \choose j}
\phi^\lambda({\bf v}_\nu^*, {\bf v}_\mu, k+i, k'+j, l+N-i, l'+M-j; z, w)]|_{z=w}$. 

\begin{lemma} \label{onefivethree}
Suppose 
$\Phi^{-+}$ and $\Phi^{+-}$ 
both exist for some 
$\lambda, \mu, \nu \geq 0$
(see (4.1)-(4.2)).
Then 
\begin{eqnarray}
\phi^\lambda({\bf v}_\nu^*, {\bf v}_\mu, 0, 0, l, 0; z, w) & \in & A_l
\end{eqnarray}
and 
\begin{eqnarray}
[ (z-w)^l \phi^\lambda({\bf v}_\nu^*, {\bf v}_\mu, 0, 0, l, 0; z, w) ]|_{z=w}
& = & 
X_{\lambda+1, \mu}^\nu
(-1)^l {\gamma \choose l} l!
w^{\lambda+ \mu -\nu},
\end{eqnarray}
where $\gamma = \frac{\lambda}{2 \varkappa}$. 
\end{lemma}

\begin{proof}
Both $\Phi^{-+}(z, w)$ and $\Phi^{+-}(z, w)$ 
are of the form 
$z^\Lambda h(t)$, 
where 
$\Lambda = \triangle(\nu) - \triangle(\lambda) - \triangle(\mu) - \triangle(1)$, 
and $h(t)$ is a function of $t = w/z$. 
The equation (4.3) that the $\Phi^{\mp \pm}$ satisfy 
can be rewritten as 
\begin{eqnarray}
\partial_w^2 \Phi 
& = & 
\left( 
\frac{1}{\varkappa} 
\frac{1}{z-w} 
+ \frac{2 (\Lambda -1 + \frac{1}{\varkappa}) }{w} 
\right) \partial_w \Phi
\\
& + & 
\left( 
\frac{1}{\varkappa}  \frac{ \triangle(\lambda) }{(z-w)^2} 
+ \frac{ 
\frac{1}{\varkappa} ( \triangle(\lambda) + \triangle( \mu) ) 
- \Lambda ( \Lambda -1 + \frac{1}{\varkappa} )
}{w^2} 
+ \frac{2 \triangle(\lambda) }{ \varkappa w (z-w) } 
\right) \Phi. 
\nonumber 
\end{eqnarray}
Note that 
\begin{eqnarray}
\phi^\lambda({\bf v}_\nu^*, {\bf v}_\mu, 0, 0, l, 0; z, w) 
& = & 
X_{1, \nu+1}^\nu X_{\lambda, \mu}^{\nu+1} 
(\partial_w^l \Phi^{-+}) \Psi^{-+} 
+ X_{\lambda, \mu}^{\nu -1} 
(\partial_w^l \Phi^{+-}) \Psi^{+-}. 
\end{eqnarray}
Using Lemma \ref{hypergeom1} and Lemma \ref{hypergeom2}, 
we can show that 
$\phi^\lambda({\bf v}_\nu^*, {\bf v}_\mu, 0, 0, 0, 0; z, w)$ 
and 
$\phi^\lambda({\bf v}_\nu^*, {\bf v}_\mu, 0, 0, 1, 0; z, w)$
satisfy (4.21)-(4.22). 
The equation (4.23) then implies that 
(4.21) is true for any $l$, 
moreover, 
it yields a recursion formula for $\tilde \varphi_{l+2}$ 
in terms of $\tilde \varphi_{l}$ and $\tilde \varphi_{l+1}$, 
where 
$\tilde \varphi_l =  [ (z-w)^l \phi^\lambda({\bf v}_\nu^*, {\bf v}_\mu, 0, 0, l, 0; z, w) ]|_{z=w}$. 
(4.22) is then verified using $\tilde \varphi_0$, 
$\tilde \varphi_1$, 
and the recursion. 
\end{proof}

\begin{lemma} \label{twelvethirty}
Under the assumptions of Lemma \ref{onefivethree}, 
we have
\begin{eqnarray}
\sum_{i=0}^N 
\sum_{j=0}^M
{N \choose i} 
{M \choose j}
\phi^\lambda({\bf v}_\nu^*, {\bf v}_\mu, k+i, k'+j, l+N-i, l'+M-j; z, w)
&  \in & A_{k+k'+l+l'}
\end{eqnarray}
and 
\begin{eqnarray}
& & 
[(z-w)^{k+k'+l+l'} 
\sum_{i=0}^N 
\sum_{j=0}^M
{N \choose i} 
{M \choose j}
\phi^\lambda({\bf v}_\nu^*, {\bf v}_\mu, k+i, k'+j, l+N-i, l'+M-j; z, w)]|_{z=w}
\\
& = &
X_{\lambda+1, \mu}^\nu 
(-1)^{l} 
{\gamma \choose k+l} (k+l)! 
{\triangle(\nu) - \triangle(\lambda+1) - \triangle(\mu) \choose N} N!
\nonumber \\
& & 
(-1)^{l'} 
{-\gamma \choose k'+l'} (k'+l')! 
{\bar \triangle(\nu) - \bar \triangle(\lambda+1) - \bar \triangle(\mu) \choose M} M!
w^{\lambda+1 + \mu - \nu -N -M}. 
\nonumber 
\end{eqnarray}
\end{lemma}

\begin{proof}
The case where $k'=l'=M=0$ 
is proved in the same manner as Corollary \ref{blah}, 
using Lemma \ref{bino} and Lemma \ref{onefivethree}.
To include all values of $k', l', M$, 
we use $D = L_{-1} + \bar L_{-1}$ to transform
the terms into the previous form.
See step 3 in the proof of Proposition \ref{extra}. 
The details are omitted.
\end{proof}

\begin{lemma} \label{onetwelve}
Lemma \ref{twelvethirty} is in fact true for any $\lambda, \mu, \nu \geq 0$. 
\end{lemma}

\begin{proof}
There are three extreme cases to take care of. 
One is $\nu = \lambda + \mu + 1$, 
the second is
$\nu = \mu- \lambda -1$, $\mu > \lambda$, 
and finally 
$\nu= \lambda - \mu -1$, $\lambda > \mu$.
They can all be proved in the same way as Lemma \ref{twelvethirty}, 
in fact the proofs are easier because
there is no hypergeometric series involved. 
The case where $\nu= \lambda - \mu -1$, $\lambda > \mu$
has all the matrix coefficients and combinations 
(4.25)-(4.26) vanish, 
which corresponds to the fact that 
$X_{\lambda+1, \mu}^{\lambda - \mu -1} = 0$.
\end{proof}

In {\bf Step 2}, we prove the following

\begin{lemma}
For $k, l, k', l', N, M \geq 0$, 
we have
\begin{eqnarray}
\sum_{i =0}^N \sum_{j=0}^M 
{N \choose i} {M \choose j}
Y(L_{-1}^{k+i} \bar L_{-1}^{k' +j} {\bf v}_1, z)
Y(L_{-1}^{l+N-i} \bar L_{-1}^{l' + M-j} {\bf v}_\lambda, w)
& \in & 
A_{k+k'+l+l'}, 
\end{eqnarray}
and 
\begin{eqnarray}
& & [(z-w)^{k+k'+l+l'} 
\sum_{i =0}^N \sum_{j=0}^M 
{N \choose i} {M \choose j}
Y(L_{-1}^{k+i} \bar L_{-1}^{k' +j} {\bf v}_1, z)
Y(L_{-1}^{l+N-i} \bar L_{-1}^{l' + M-j} {\bf v}_\lambda, w)
]|_{z=w}
\\
& = & 
(-1)^{l+l'} 
{\gamma \choose k+l} (k+l)! 
{-\gamma \choose k'+l'} (k' +l')!
[\sum_{i =0}^N \sum_{j=0}^M 
{N \choose i} {M \choose j}
Y(L_{-1}^{i} \bar L_{-1}^{j} {\bf v}_1, z)
Y(L_{-1}^{N-i} \bar L_{-1}^{ M-j} {\bf v}_\lambda, w)
]|_{z=w}. 
\nonumber 
\end{eqnarray}
Everything is understood on the level of matrix coefficients. 
\end{lemma}

\begin{proof}
Lemma \ref{onetwelve} implies that this is true for 
matrix coefficients determined by the highest weight vectors
$({\bf v}_\nu^*, {\bf v }_\mu)$. 
The general statement is proved by induction, 
the same kind as implemented in the proof of Proposition \ref{extra}. 
\end{proof}

Since ${\bf v}_{\lambda+1} = {\bf v}_1(-1) {\bf v}_\lambda$ 
and the fields $Y({\bf v}_1, z)$, 
$Y({\bf v}_\lambda, w)$ commute, 
the Reconstruction Theorem \ref{reconstruction} implies that 
$$
Y({\bf v}_{\lambda+1}, w) = Y({\bf v}_1, z) Y({\bf v}_\lambda, w)|_{z = w}. 
$$
In general, 
since 
$$
L_{-1}^N \bar L_{-1}^M {\bf v}_{\lambda+1} =
\sum_{i=0}^N \sum_{j=0}^M 
{N \choose i} {M \choose j}
(L_{-1}^i \bar L_{-1}^j {\bf v}_1 ) (-1)
(L_{-1}^{N-i} \bar L_{-1}^{M-j} {\bf v}_\lambda),
$$
Theorem \ref{reconstruction} implies that 
\begin{eqnarray}
Y(L_{-1}^N \bar L_{-1}^M {\bf v}_{\lambda+1}, w)
& = & 
\sum_{i=0}^N \sum_{j=0}^M
{N \choose i} {M \choose j}
:Y(L_{-1}^i \bar L_{-1}^j {\bf v}_1, w) Y( L_{-1}^{N-i} \bar L_{-1}^{M-j} {\bf v}_\lambda, w):. 
\end{eqnarray}
Without the normal ordering, 
$$
\sum_{i=0}^N \sum_{j=0}^M
{N \choose i} {M \choose j}
Y(L_{-1}^i \bar L_{-1}^j {\bf v}_1, z) Y( L_{-1}^{N-i} \bar L_{-1}^{M-j} {\bf v}_\lambda, w) 
$$
is shown to be regular
(this is exactly what (4.27) says when $k=k'=l=l'=0$), 
hence
we rewrite (4.29) as
\begin{eqnarray}
Y(L_{-1}^N \bar L_{-1}^M {\bf v}_{\lambda+1}, w)
& = & 
[\sum_{i=0}^N \sum_{j=0}^M
{N \choose i} {M \choose j}
Y(L_{-1}^i \bar L_{-1}^j {\bf v}_1, z) Y( L_{-1}^{N-i} \bar L_{-1}^{M-j} {\bf v}_\lambda, w)]|_{z=w}. 
\end{eqnarray}

\begin{corollary} \label{ha}
For any $N, M \geq 0$, $\mu, \nu \geq 0$, 
we have 
\begin{eqnarray}
& & ({\bf v}_\nu^*, Y(L_{-1}^N \bar L_{-1}^M {\bf v}_{\lambda+1}, w) {\bf v}_\mu)
\\
& = &
X_{\lambda+1, \mu}^\nu
{\triangle(\nu) - \triangle(\lambda+1) - \triangle(\mu) \choose N} N!
{\bar \triangle(\nu) - \bar \triangle(\lambda+1) - \bar \triangle(\mu) \choose M} M!
w^{\lambda+1 + \mu - \nu -N - M}. 
\nonumber 
\end{eqnarray}
\end{corollary}

\begin{proof}
This follows from setting $k=k'=l=l'=0$ in (4.26). 
\end{proof}

\begin{corollary} \label{ti}
For any $p \in \mathbb Z$, $N, M \geq 0$, 
we have 
\begin{eqnarray}
[L_p, Y(L_{-1}^N \bar L_{-1}^M {\bf v}_{\lambda+1}, w)] 
& =  &
\sum_{q \geq 0} {p+1 \choose q} w^{p-q+1} 
Y(L_{q-1} L_{-1}^N \bar L_{-1}^M {\bf v}_{\lambda+1}, w), 
\end{eqnarray}
\begin{eqnarray}
[\bar L_p, Y(L_{-1}^N \bar L_{-1}^M {\bf v}_{\lambda+1}, w)] 
& =  &
\sum_{q \geq 0} {p+1 \choose q} w^{p-q+1} 
Y(\bar L_{q-1} L_{-1}^N \bar L_{-1}^M {\bf v}_{\lambda+1}, w). 
\end{eqnarray}
\end{corollary}

\begin{proof}
This follows from
(4.30) and 
(special cases of )
(4.28). 
The computations are rather tedious, 
therefore omitted.
\end{proof}

\noindent
{\bf Step 3. }
The vertex operators
$Y(L_{-1}^N \bar L_{-1}^M {\bf v}_{\lambda+1}, z)$
are determined by (4.31)-(4.33) completely. 
The same formulas would hold if we were to define 
$Y(L_{-1}^N \bar L_{-1}^M {\bf v}_{\lambda+1}, z)$
by the pairing of intertwining operators
with structure constants 
$X_{\lambda+1, \mu}^\nu$, 
hence
it is indeed true. 
Once this is established, 
it is not difficult to see that the vertex operator of 
any element of $W(\lambda+1)$ is obtained this way. 

\medskip

The induction part is now complete. 
We summarize the main theorem.

\begin{theorem}
Let $\varkappa \in \mathbb C \backslash \mathbb Q$
and set 
$c =  13 - 6 \varkappa - 6 \varkappa^{-1}$; 
$\bar c =  13 + 6 \varkappa + 6 \varkappa^{-1}$. 
Also define 
$\triangle(\lambda) = \frac{\lambda (\lambda+2)}{4\varkappa} - \frac{\lambda}{2}$ and 
$\bar \triangle(\lambda) = - \frac{\lambda (\lambda+2)}{4\varkappa} - \frac{\lambda}{2}$
for $\lambda \in \mathbb N$. 
Let $L(\triangle(\lambda), c)$ 
(resp. $L(\bar \triangle(\lambda), \bar c)$)
denote the irreducible highest weight representation 
of the Virasoro Lie algebra with highest weight 
$(\triangle(\lambda), c)$
(resp. $(\bar \triangle(\lambda), \bar c)$).  

Then 
the $\text{Vir}_c \oplus \text{Vir}_{\bar c}$-module 
\begin{eqnarray}
W = \bigoplus_{\lambda \in \mathbb N} L( \triangle(\lambda), c) \otimes L( \bar \triangle(\lambda), \bar c)
\nonumber
\end{eqnarray}
is a vertex operator algebra of rank $26$. 
Furthermore, 
the vertex operators are given as follows: 
set 
\begin{eqnarray}
W(\lambda) & = & L( \triangle(\lambda), c) \otimes L( \bar \triangle(\lambda), \bar c); 
\nonumber
\end{eqnarray}
then for any $v \otimes v' \in W(\lambda)$, $ u \otimes u' \in W(\mu)$, 
we have 
\begin{eqnarray}
Y(v \otimes v', z) u \otimes u' 
& = & 
\sum_{\nu \in \mathbb N} X_{\lambda, \mu}^\nu
\Phi_{\lambda, \mu}^\nu(v, z) u \otimes \Psi_{\lambda, \mu}^\nu (v', z) u' 
\nonumber 
\end{eqnarray}
where $\Phi_{\lambda, \mu}^\nu(\cdot, z)$ (resp. $ \Psi_{\lambda, \mu}^\nu (\cdot, z)$ ) 
is the intertwining operator of type 
${L(\triangle(\nu), c) \choose L(\triangle(\lambda), c) \, L(\triangle(\mu), c)}$
(resp. ${L(\bar \triangle(\nu), \bar c) \choose L(\bar \triangle(\lambda), \bar c) \, L(\bar \triangle(\mu), \bar c)}$)
such that 
\begin{eqnarray}
\Phi_{\lambda, \mu}^\nu(v, z) 
& = & 
z^{\triangle(\nu) -\triangle(\lambda) -\triangle(\mu)} \sum_{i \in \mathbb Z} v(n)z^{-n-1}
\nonumber 
\end{eqnarray}
\begin{eqnarray}
\text{(resp. }  \Psi_{\lambda, \mu}^\nu (v', z) 
& = & 
z^{\bar \triangle(\nu) -\bar \triangle(\lambda) -\bar \triangle(\mu)} \sum_{i \in \mathbb Z} v'(n)z^{-n-1})
\nonumber 
\end{eqnarray}
for any $v \in L(\triangle(\lambda), c)$ (resp. $v' \in L(\bar  \triangle(\lambda), \bar c)$) 
and 
$v_{\triangle(\lambda), c} (-1) v_{\triangle(\mu), c} = v_{\triangle(\nu), c}$
(resp. 
$v_{\bar \triangle(\lambda), \bar c} (-1) v_{\bar \triangle(\mu), \bar c} = v_{\bar \triangle(\nu), \bar c}$). 
The constants $X_{\lambda, \mu}^\nu$ are given in Definition \ref{structureconstants}. 
\end{theorem}

\section{Hypergeometric identities}

\noindent {\bf \emph{Proof of Lemma \ref{hypergeom1}: }} 
The hypergeometric differential equation 
\begin{eqnarray}
t(1-t) F''(t) + [ c- (a+b+1) t ] F'(t) - ab F & = & 0
\end{eqnarray}
has three regular singularities at $0$, $1$, and $\infty$. 
Local solutions around regular singularities can be found by the Frobenius method. 
Suppose none of $c$, $c-a-b$, or $a-b$ is an integer, 
linearly independent solutions around $0$, $1$, and $\infty$ are given by 
\begin{eqnarray}
\Phi_0 & = & 
\left( \begin{array}{c}
{}_2F_1(a, b, c; t) \\
t^{1-c} {}_2F_1(a-c+1, b-c+1, 2-c; t)
\end{array}
\right) 
\nonumber \\
\Phi_1 & = & 
\left( \begin{array}{c}
{}_2F_1(a, b, a+b-c+1; 1-t) \\
(1-t)^{c-a-b} {}_2F_1(c-a, c-b, c-a-b+1; 1-t) 
\end{array}
\right) 
\nonumber \\
\Phi_\infty & = & 
\left( \begin{array}{c}
t^{-a} {}_2F_1(a, a-c+1, a-b+1; \frac{1}{t}) \\
t^{-b} {}_2F_1(b, b-c+1, b-a+1; \frac{1}{t})
\end{array}
\right). 
\nonumber
\end{eqnarray}
By analytic continuation, 
each set of solutions can be expressed as linear combinations of another set. 
That is $\Phi_0 = M_1 \Phi_1$ and $\Phi_0 = M_\infty \Phi_\infty$, 
where 
the connection matrices $M_1$ and $M_\infty$ are as follows: 
\begin{eqnarray}
M_1 & = & 
\left( \begin{array}{cc}
\frac{\Gamma (c) \Gamma (c-a-b)}{\Gamma (c-a) \Gamma (c-b)} &
\frac{\Gamma(c)\Gamma (a+b-c) }{\Gamma (a) \Gamma (b)}
\\
\frac{\Gamma(2-c) \Gamma(c-a-b)}{\Gamma(1-a) \Gamma(1-b)} &
\frac{\Gamma(2-c) \Gamma(a+b-c)}{\Gamma(a-c+1)\Gamma(b-c+1)}
\end{array}
\right)
\nonumber \\
M_\infty & = & 
\left( \begin{array}{cc}
e^{- i \pi a} 
\frac{\Gamma (c) \Gamma (b-a)}{\Gamma (c-a) \Gamma (b)} &
e^{- i \pi b}
\frac{\Gamma(c)\Gamma (a-b) }{\Gamma (c-b) \Gamma (a)}
\\
e^{-i \pi (a-c+1)} 
\frac{\Gamma(2-c) \Gamma(b-a)}{\Gamma(b-c+1) \Gamma(1-a)} &
e^{- i \pi (b-c+1)} 
\frac{\Gamma(2-c) \Gamma(a-b)}{\Gamma(a-c+1)\Gamma(1-b)}
\end{array}
\right)
\nonumber 
\end{eqnarray}
(see e.g. [S]). 
The dual equation 
\begin{eqnarray}
t(1-t) F''(t) + [ -c + (a+b+1) t ] F'(t) - a(b+2) F & = & 0
\nonumber 
\end{eqnarray}
has the following solutions 
\begin{eqnarray}
\Psi_0 & = & 
\left( \begin{array}{c}
{}_2F_1(-a, -b-2, -c; t) \\
t^{1+c} {}_2F_1(-a+c+1, -b+c-1, 2+c; t)
\end{array}
\right) 
\nonumber \\
\Psi_1 & = & 
\left( \begin{array}{c}
{}_2F_1(-a, -b-2, c-a-b-1; 1-t) \\
(1-t)^{a+b-c+2} {}_2F_1 (a-c, b-c+2, a+b-c+3; 1-t)
\end{array}
\right) 
\nonumber \\
\Psi_\infty & = & 
\left( \begin{array}{c}
t^a {}_2F_1(-a, -a+c+1, -a+b+3; \frac{1}{t}) \\
t^{b+2} {}_2F_1(-b-2, -b+c-1, a-b-1; \frac{1}{t})
\end{array}
\right). 
\nonumber
\end{eqnarray}
The connection matrices relating them are as follows: 
\begin{eqnarray}
N_1 & = & 
\left( \begin{array}{cc}
\frac{\Gamma (-c) \Gamma (a+b-c+2)}{\Gamma (a-c) \Gamma (b-c+2)} &
\frac{\Gamma(-c )\Gamma ( -a -b +c -2 ) }{\Gamma (- a) \Gamma (- b -2 )}
\\
\frac{\Gamma(2 + c) \Gamma(a+b-c+2 )}{\Gamma(a+1) \Gamma(b+3)} &
\frac{\Gamma(2 + c) \Gamma(-a -b +c -2)}{\Gamma( -a +c +1)\Gamma( -b + c -1)}
\end{array}
\right)
\nonumber \\
N_\infty & = & 
\left( \begin{array}{cc}
e^{ i \pi a} 
\frac{\Gamma ( -c) \Gamma (a - b - 2)}{\Gamma (a-c ) \Gamma (- b -2)} &
e^{ i \pi b}
\frac{\Gamma( - c)\Gamma ( -a + b+2 ) }{\Gamma (b -c +2) \Gamma (-a)}
\\
e^{i \pi ( a-c -1)} 
\frac{\Gamma(2 + c) \Gamma( a -b -2)}{\Gamma( - b +c  -1) \Gamma(1+a )} &
e^{i \pi ( b - c + 1)} 
\frac{\Gamma(2 + c) \Gamma( -a + b+2 )}{\Gamma( - a + c+1)\Gamma(b +3)}
\end{array}
\right)
\nonumber 
\end{eqnarray}
so that 
$\Psi_0 = N_1 \Psi_1$ and $\Psi_\infty = N_\infty \Psi_\infty$. 
Let
\begin{eqnarray}
Q & = & 
\left( \begin{array}{cc}
1 & 0 
\\
0 & 
\frac{a b (b+1) (b+2) (a-c) }{ (c-1) c^2 (c+1) (c-b -1) }
\end{array}
\right). 
\nonumber
\end{eqnarray}
The left-hand-side of (3.20) is 
$\Phi_0^t Q \Psi_0$, 
call it $f(t)$. 
$f(t)$ is a single-valued holomorphic function in the disk $|t| <1$. 
When it is analytically continued to a punctured neighborhood of $1$, 
we get a possibly multi-valued function. 
Denote it again by $f$; 
$f$ may have non-trivial monodromy around $1$. 
Since 
$\Phi_0 = M_1 \Phi_1$ and $\Psi_0 = N_1 \Psi_1$, 
we have 
\begin{eqnarray}
f(t) & = & \Phi_1^t M_1^t Q N_1 \Psi_1. 
\nonumber 
\end{eqnarray}
The $(1, 1)$-th entry of the matrix $Q_1 = M_1' Q N_1$ is 
\begin{eqnarray}
& & \frac{\Gamma (c) \Gamma (c-a-b)}{\Gamma (c-a) \Gamma (c-b)} 
\frac{\Gamma (-c) \Gamma (a+b-c+2)}{\Gamma (a-c) \Gamma (b-c+2)} 
 \\
& + & 
\frac{a b (b+1) (b+2) (a-c) }{ (c-1) c^2 (c+1) (c-b -1) }
\frac{\Gamma(2-c) \Gamma(c-a-b)}{\Gamma(1-a) \Gamma(1-b)} 
\frac{\Gamma(2 + c) \Gamma(a+b-c+2 )}{\Gamma(a+1) \Gamma(b+3)}. 
\nonumber
\end{eqnarray}
Since 
\begin{eqnarray}
\Gamma(z +1) = z \Gamma(z), & & 
\Gamma (z) \Gamma (1-z) = \frac{\pi}{\sin \pi z}, 
\nonumber 
\end{eqnarray}
(5.2) is simplified to 
\begin{eqnarray}
& & \frac{ (a+b -c+1) (a-c) }{ c (c-b-1)} 
\frac{ \sin \pi (a-c) \sin \pi (c-b) + \sin \pi a \sin \pi b} {\sin \pi c \sin \pi (a+b-c) }. 
\nonumber
\end{eqnarray}
Furthermore, 
using the identity 
\begin{eqnarray}
\sin (\alpha - \gamma) \sin (\gamma - \beta) + \sin \alpha \sin \beta 
& = &
\sin \gamma \sin (\alpha + \beta - \gamma ), 
\nonumber 
\end{eqnarray}
we get 
\begin{eqnarray}
Q_1 [1, 1]  & = &  \frac{ (a+b -c+1) (a-c) }{ c (c-b-1)}. 
\nonumber 
\end{eqnarray}
Similar calculations show that 
\begin{eqnarray}
Q_1 [1, 2]  & = & Q_1 [2, 1] = 0
\nonumber \\
Q_1 [ 2, 2] & = & \frac{a b(b+1)(b+2) }{c (a+b-c) (a+b-c+1)(a+b-c+2)}. 
\nonumber 
\end{eqnarray}
Hence, 
\begin{eqnarray}
f(t) & = & 
\frac{ (a+b -c+1) (a-c) }{ c (c-b-1)} \Phi_1' 
\left( \begin{array}{cc}
1 & 0 \\
0 & \frac{a b (b+1)(b+2) (c-b-1) }{(a+b-c)(a+b-c+1)^2 (a+b-c+2) (a-c)}
\end{array}
\right)
\Psi_1. 
\nonumber 
\end{eqnarray}
In particular,  
$f$ is single-valued and holomorphic in a neighborhood of 1. 
Hence, 
$f$ is an entire function on the complex plane. 
Around the point $\infty$, 
we can again check that 
\begin{eqnarray}
f(t) & = & 
\Phi_\infty^t M_\infty^t  Q  N_\infty \Psi_\infty
\nonumber 
\end{eqnarray}
where 
\begin{eqnarray} 
Q_\infty & = & M_\infty^t Q  N_\infty 
= \left( \begin{array}{cc}
\frac{b (b+1)(b+2) (a-c) }{c ( a- b) (a-b-1) (a-b-2) } & 0 \\
0 &  \frac{a (a-b-1) }{c(c-b-1)} 
\end{array}
\right).
\nonumber 
\end{eqnarray}
Hence, 
$f$ has a pole of order $2$ at $\infty$. 
This implies that $f(t)$ is a polynomial of degree $2$; 
the (first) three coefficients are computed directly. 
\hfill 
$\square$

\medskip

\noindent {\bf \emph{Proof of Lemma \ref{hypergeom2}:}} 
Consider the function 
$f_1(t) = ( \frac{\partial \Phi_0}{\partial t} )' Q \Psi_0$. 
It is clear that $f_1$ is a (single-valued) holomorphic function in the disk $|t| <1$. 
By analytic continuation, 
near the point $1$,
$f_1 = (\frac{ \partial \Phi_1}{\partial t} )' Q_1 \Psi_1$ 
is also single-valued and holomorphic. 
At the point $\infty$, 
$f_1 =  (\frac{ \partial \Phi_\infty}{\partial t} )' Q_\infty \Psi_\infty$
has a pole of order 1. 
Hence, 
$f_1$ is a polynomial of degree 1. 
Explicit calculation shows that 
\begin{eqnarray}
f_1 & = & 
\frac{ab}{c} \left( 1 - \frac{a-b-1}{c-b-1} t \right). 
\nonumber 
\end{eqnarray}
The left-hand side of (3.34) is now
$ \frac{t}{c-1} f_1 + f$. 
\hfill 
$\square$

\medskip 

\noindent {\bf \emph{Proof of Lemma \ref{hypergeom3}:}} 
The function $f_2(t) = (\frac{ \partial^2 \Phi_0}{\partial t^2} )' Q \Psi_0$
is holomorphic at $0$, $1$, and $\infty$; 
hence it is a constant. 
Explicitly, 
\begin{eqnarray} 
f_2(t) & = & 
\frac{ ab(b+1) (a-b-1) }{c (c-b-1) }. 
\nonumber 
\end{eqnarray} 
The left-hand side of (3.38) is now  
\begin{eqnarray} 
& & \frac{t^2}{(c-2)(c-1)} f_2 + \frac{2t}{c-2} f_1 + f.  
\nonumber 
\end{eqnarray} 

In general, 
we can consider the function
\begin{eqnarray} 
f_n & = & 
( \frac{ \partial^n \Phi_0 }{\partial t^n} )' Q \Psi_0. 
\nonumber
\end{eqnarray} 
When $n \geq 3$, 
$f_n$ has a pole of order $n-2$ at $0$, 
a pole of order $n-2$ at $1$, 
and a zero of order $n-2$ at $\infty$. 
Hence, 
$f_n$ is of the form 
$$
\frac{A_{n-2} }{t^{n-2} } + \cdots + \frac{A_1}{t} + \frac{B_{n-2} }{(1-t)^{n-2} } + \cdots + \frac{B_1}{1-t}.
$$
This gives a family of identities of hypergeometric series. 
We write down two more of them. 

\begin{lemma}
\begin{eqnarray}
{}_2F_1(a, b, c, t) {}_2F_1(-a, -b-2, -c, t)  & & \\
+ \frac{3ab}{c(c-3)}t {}_2F_1(a+1, b+1, c+1, t) {}_2F_1(-a, -b-2, -c, t) & & \nonumber \\
+ \frac{3(a)_2(b)_2}{(c)_2(c-3)_2} t^2 {}_2F_1(a+2, b+2, c+2, t) {}_2F_1(-a, -b-2, -c, t) & & \nonumber \\
+ \frac{(a)_3(b)_3}{(c-3)_6} t^3  {}_2F_1(a+3, b+3, c+3, t) {}_2F_1(-a, -b-2, -c, t) & & \nonumber \\
+\frac{a(b)_3(a-c)_4 (c-b-3)_2}{(c-4)_5(c-3)_5} t^5 
 {}_2F_1(a-c+4, b-c+4, 5-c, t) {}_2F_1(-a+c+1, -b+c-1, c+2, t) & & \nonumber \\
= 1 + \frac{a(3b-2c+6)}{c(c-3)} t + \frac{a(a-b-1)(c-b-3)_2}{(c-3)_4} t^2 + \frac{a(b)_3}{(c-3)_4} \frac{t^2}{1-t} \nonumber
\end{eqnarray}
\end{lemma}

\begin{lemma}
\begin{eqnarray}
{}_2F_1(a, b, c, t) {}_2F_1(-a, -b-2, -c, t)  & & \\
+ \frac{4ab}{c(c-4)}t {}_2F_1(a+1, b+1, c+1, t) {}_2F_1(-a, -b-2, -c, t) & & \nonumber \\
+ \frac{6(a)_2(b)_2}{(c)_2(c-4)_2} t^2 {}_2F_1(a+2, b+2, c+2, t) {}_2F_1(-a, -b-2, -c, t) & & \nonumber \\
+ \frac{4(a)_3(b)_3}{(c)_3(c-4)_3} t^3  {}_2F_1(a+3, b+3, c+3, t) {}_2F_1(-a, -b-2, -c, t) & & \nonumber \\
+ \frac{(a)_4(b)_4}{(c-4)_8} t^4  {}_2F_1(a+4, b+4, c+4, t) {}_2F_1(-a, -b-2, -c, t) & & \nonumber \\
+\frac{a(b)_3(a-c)_5 (c-b-4)_3}{(c-5)_6(c-4)_6} t^6 
 {}_2F_1(a-c+5, b-c+5, 6-c, t) {}_2F_1(-a+c+1, -b+c-1, c+2, t) & & \nonumber \\
= 1 + \frac{a(4b-2c+8)}{c(c-4)} t + \frac{a(a-b-1)(c-b-4)_3}{(c-4)_5} t^2 & & \nonumber  \\
+ \frac{a(b)_3(a-b+3c-7)}{(c-4)_5} \frac{t^2}{1-t} 
+ \frac{a(b)_3(a +b -c+3)}{(c-4)_5} \frac{t^3}{(1-t)^2} & & \nonumber
\end{eqnarray}
\end{lemma}



\end{document}